\documentclass[11pt]{article}
\usepackage[tbtags]{amsmath}
\usepackage{cases}
\usepackage{mathrsfs}
\usepackage{amsfonts}
\usepackage{authblk}
\usepackage{amsmath,amsthm,amssymb}
\usepackage{cite}

\usepackage{color}

\theoremstyle{plain}
\newtheorem{theorem}{Theorem}[section]
\newtheorem{lemma}{Lemma}[section]
\newtheorem{remark}{Remark}[section]
\newtheorem{proposition}{Proposition}[section]

\newif \ifLastSection \LastSectionfalse

\numberwithin{equation}{section}

\allowdisplaybreaks

\begin{document}

\title{{\bf {Convergence to nonlinear diffusion waves for solutions of $M_1$ model}}}

\author[1]{Nangao Zhang}
\author[2]{Changjiang Zhu\thanks{Corresponding author. \authorcr Email addresses: mazhangngmath@mail.scut.edu.cn
 (zhang), machjzhu@scut.edu.cn (zhu)
.}}
\affil[1,2]{ \normalsize  School of Mathematics, South China University of Technology, Guangzhou 510641, P.R. China}

\date{}

\maketitle

\textbf{{\bf Abstract:}} In this paper, we are concerned with the asymptotic behavior of solutions of $M_1$ model proposed in the radiative transfer fields.
Starting from this model, combined with the compressible Euler equation with damping, we introduce a more general system.
We rigorously prove that the solutions to the Cauchy problem of this system globally exist and time-asymptotically converge to the shifted nonlinear diffusion waves whose profile is self-similar solution to the corresponding parabolic equation governed by the classical Darcy's law.
Moreover, the optimal convergence rates are also obtained.
Compared with previous results obtained by Nishihara, Wang and Yang in \cite{Nishihara-Wang-Yang2000}, we have a weaker and more general condition on the initial data, and the conclusions are more sharper. The approach adopted in the paper is the technical time-weighted energy estimates with the  Green function method together.

\bigbreak \textbf{{\bf Key Words}:} $M_1$ model, Darcy's law, nonlinear diffusion waves, time-weighted energy estimates, Green function method.

\bigbreak \textbf{{\bf AMS Subject Classification}:} 85A25, 35L65, 35B40.

\section{Introduction and main results}\label{S1}

Radiative transfer has a huge influence on the hydrodynamic flow in applications such as superorbital atmospheric re-entry, fires or astrophysics. In such regimes, it is important to have a good prediction of the radiative variables. However, solving the full radiative transfer equation is too expensive. It is hence necessary to develop other models for radiative that are cheap, yet accurate enough to give good predictions of the radiative effects. In this case, the $M_{1}$ model is an interesting choice (cf. \cite{Berthon-Charrier2003, Berthon-Dubois2010}). In the present paper, we just consider the scattering part and we omit the role played by the temperature, then the corresponding simplified model reads as follows (cf. \cite{Berthon-Charrier-Dubroca2007, Goudon-Lin2013}):
\begin{equation}\label{1.1a}
\left\{\begin{array}{l}
\partial_{t}\rho+c\nabla \cdot (\rho u)=0,\\[2mm]
  \partial_{t}(\rho u)+ c\nabla \cdot P(\rho, u)=-c\sigma\rho u.
 \end{array}
        \right.
\end{equation}
Here, the unknown function $\rho=\rho(x,t) \geq 0$ denotes the radiative energy, and $u=u(x,t) \in \mathbb{R}^{n} (1\leq n\leq3)$ denotes the normalized radiative flux. The positive constants $c$ and $\sigma$ denote the speed of the light and the opacity respectively.
Concerning the radiative pressure $P(\rho, u)$, it is given by
\begin{equation}\label{1.2a}
P(\rho, u)=\frac{1}{2}\left((1-\chi(u)) \mathbb{I}_{n}+(3 \chi(u)-1) \frac{u \otimes u}{|u|^{2}}\right) \rho,
\end{equation}
with
\begin{equation}\label{1.3a}
  \chi(u)=\frac{3+4 |u|^{2}}{5+2 \sqrt{4-3|u|^{2}}},
\end{equation}
 where $\mathbb{I}_{n}$ is the identity matrix of order $n$ and $\left|u\right| \leq 1$.

In this paper, we shall restrict ourselves to the one-dimensional case. We set $c=1$ without loss of generality, then \eqref{1.1a} can be rewritten as
\begin{equation}\label{1.4a}
\left\{\begin{array}{l}
 \rho_{t} +\left(\rho u \right)_{x}=0,\\[2mm]
 \displaystyle \left(\rho u \right)_{t}+\left(\frac{\rho}{3}\right)_{x}+\left(\frac{2\rho u^{2}}{2+\sqrt{4-3u^{2}}} \right)_{x}=-\sigma \rho u,
 \end{array}
        \right.
\end{equation}
with the following initial data and the far field behaviors
\begin{equation}\label{1.5a}
 (\rho,u)|_{t=0}=(\rho_0,u_0)(x) \rightarrow (\rho_{\pm}, u_{\pm}), \quad \mbox{as} \quad x \rightarrow \pm \infty \quad \mbox{with} \quad \rho_{+} \neq \rho_{-},
\end{equation}
where $\rho_{\pm}$ and $u_{\pm}$ are the constant states.

We are interested in the large time behavior of solutions to the Cauchy problem \eqref{1.4a}-\eqref{1.5a}. Suppose that $\rho \geq C > 0$, then it is more convenient to use the Lagrangian coordinates to explore this system. We consider the coordinate transformation as follows:
\begin{equation}\notag
 x \Rightarrow \int_{(0,0)}^{(x, t)} \rho(y, s) \mathrm{d} y-(\rho u)(y, s) \mathrm{d} s, \quad t \Rightarrow \tau ,
\end{equation}
 and we still denote the Lagrangian coordinates by $(x, t)$ for simplicity.

 Let $v=\frac{1}{\rho}$, then the Cauchy problem \eqref{1.4a}-\eqref{1.5a} can be transformed as the following form
\begin{equation}\label{1.7a}
 \left\{\begin{array}{l}
 v_{t}-u_{x}=0,\\[2mm]
\displaystyle  u_{t}+\left(\frac{1}{3v} \right)_{x}-\left(\frac{u^{2}\sqrt{4-3u^{2}}}{v(2+\sqrt{4-3u^{2}})} \right)_{x}=- \sigma u,
 \end{array}
        \right.
\end{equation}
with initial data
\begin{equation}\label{1.8a}
  (v,u)|_{t=0}=(v_0,u_0)(x) \rightarrow (v_{\pm}, u_{\pm}), \quad \mbox{as} \quad x \rightarrow \pm \infty \quad \mbox{with} \quad v_{+} \neq v_{-}.
\end{equation}

Due to its complexity, the study on \eqref{1.1a} is quite limited and far from being well. The global existence of smooth solutions with small initial data has been studied by many authors (see \cite{Goudon-Lin2013,Kawashima1983,Li-Peng-Zhao2021,Nishida1978} and references therein), and some numerical methods are also considered in \cite{Berthon-Charrier-Dubroca2007,Cordier-Buet2004}. However, there are very few studies on the large time behavior of solutions to $M_{1}$ model, to the best of our knowledge. In physics, the damping effects usually causes the dynamical system to possess the nonlinear diffusive phenomena, such interesting phenomena for 1-D compressible Euler equations with damping was firstly observed by Hsiao and Liu in \cite{Hsiao-Liu1992}. Here we are also mainly concerned with the nonlinear diffusive phenomena of \eqref{1.7a}-\eqref{1.8a}.

Considering the relationship between \eqref{1.7a} and the compressible Euler equations with damping, we expect to obtain more general results including these two systems. So in this paper, we prefer to consider the following more general system
\begin{equation}\label{1.1}
 \left\{\begin{array}{l}
 v_{t}-u_{x}=0,\\[2mm]
  u_{t}+p(v)_{x}=-\alpha u+(g(u)f(v))_{x}, \quad (x,t) \in \mathbb{R} \times \mathbb{R}^{+},
 \end{array}
        \right.
\end{equation}
with initial data
\begin{equation}\label{1.2}
  (v,u)|_{t=0}=(v_0,u_0)(x) \rightarrow (v_{\pm}, u_{\pm}), \quad \mbox{as} \quad x \rightarrow \pm \infty \quad \mbox{and}~~v_{+} \neq v_{-} .
\end{equation}
 Here $u=u(x,t) ~\mbox{and}~ v=v(x,t)>0:\mathbb{R}\times (0,\infty) \rightarrow \mathbb{R}$ are unknown variables, $p$ is a smooth function of $v$ with $p>0$, $g$ and $f$ are smooth function of $u$ and $v$, respectively. $v_0(x)$ and $u_0(x)$ are the given initial data, damping constant $\alpha > 0$, $v_{\pm}>0$ and $u_{\pm}$ are constants.

Let us recall some known results about the above system now. For $(g(u)f(v))_{x}\equiv 0$, the system \eqref{1.1} becomes the compressible Euler equations with linear damping
\begin{equation}\label{1.3}
 \left\{\begin{array}{l}
 v_{t}-u_{x}=0,\\[2mm]
  u_{t}+p(v)_{x}=-\alpha u,\quad (x,t) \in \mathbb{R} \times \mathbb{R}^{+}.
 \end{array}
        \right.
\end{equation}
The global existence and asymptotic behavior of the solutions to the Cauchy problem for \eqref{1.3} has been extensively studied (see \cite{Fang-Xu2009, Jiu-Zheng2012, Kawashima1983, Wang-Yang2001} and references therein). Among them, Hsiao and Liu in \cite{Hsiao-Liu1992} firstly showed the solutions $(v,u)$ of \eqref{1.3} tended time-asymptotically to the nonlinear diffusion waves $(\bar{v},\bar{u})$ of the system
\begin{equation}\notag
 \left\{\begin{array}{l}
 \bar{v}_{t}=-\frac{1}{\alpha}p(\bar{v})_{xx},\\[2mm]
  p(\bar{v})_{x}=-\alpha \bar{u},
 \end{array}
   \right.	
\end{equation}
in the sense
 \begin{equation}\notag
 \left\|(v-\bar{v}, u-\bar{u})(t) \right\|_{L^{\infty}} \leq C(t^{-\frac{1}{2}},t^{-\frac{1}{2}}),	
 \end{equation}
 when initial perturbation is small belonging to $H^{3} \times H^{2}$. Then, by taking more detailed but elegant energy estimates, Nishihara in \cite{Nishihara1996} successfully improved the convergence rates as
\begin{equation}\notag
 \left\|(v-\bar{v}, u-\bar{u})(t) \right\|_{L^{\infty}} \leq C(t^{-\frac{3}{4}},t^{-\frac{5}{4}}),	
 \end{equation}
 provided that small initial perturbation belongs to $H^{3} \times H^{2}$. Subsequently, when the small initial perturbation belonged to $(H^{3} \cap L^{1}) \times (H^{2} \cap L^{1})$, by constructing an appropriate approximate Green's function and using energy methods, Nishihara, Wang and Yang in \cite{Nishihara-Wang-Yang2000} further improved the convergence rates as
  \begin{equation}\notag
 \left\|(v-\bar{v}, u-\bar{u})(t) \right\|_{L^{\infty}} \leq C(t^{-1},t^{-\frac{3}{2}}),	
 \end{equation}
 which is optimal in the sense comparing with the heat equation.
 These conclusions require that both the initial disturbance and the wave strength around a particular diffusion wave are suitably small, some of these restrictions were later partially relaxed by Zhao in \cite{Zhao2001}. Later, Mei in \cite{Mei2010} pointed out that the best asymptotic profiles are the solutions for the corresponding nonlinear diffusion equation with some specific selected initial data, and obtained that the convergence rates to the profile is in the form of
\begin{equation}\notag
 \left\|(v-\bar{v}, u-\bar{u})(t) \right\|_{L^{\infty}} \leq C(t^{-\frac{3}{2}}\ln t,t^{-2}\ln t),	
 \end{equation}
 provided that small initial perturbation belongs to $(H^{3} \cap L^{1}) \times (H^{2} \cap L^{1})$.
 For other studies related to \eqref{1.3} with nonlinear damping or vacuum, and so on, we can refer to these interesting works (\cite{Hsiao-Liu1993, Hsiao-Luo1996, Hsiao-Tang1995, Huang-Marcati-Pan2005, Huang-Pan2003, Huang-Pan2006, Huang-Pan-Wang2011, Marcati-Mei-Rubino, Marcati-Nishihara, Mei2009, Nishihara2003, Zhu2003, Zhu-Jiang2006}) and references therein.

 When $\alpha=\sigma$, $p(v)=\frac{1}{3v}, g(u)=\frac{u^{2}\sqrt{4-3u^{2}}}{2+\sqrt{4-3u^{2}}}$ and $f(v)=\frac{1}{v}$, the system \eqref{1.1} can be reduced to $M_{1}$ model \eqref{1.7a} which we will study in the following.
Inspired by these preceding results, in the present paper, we will discuss the convergence to nonlinear diffusion waves for solutions of \eqref{1.1}-\eqref{1.2}, and we will obtain a sharper result which indeed improves those in Nishihara, Wang and Yang \cite{Nishihara-Wang-Yang2000} (See Remark \ref{R1}-\ref{R2}).

As in \cite{Hsiao-Liu1992,Nishihara1996}, the solutions of \eqref{1.1} time-asymptotically behave as those of Darcy's law
\begin{equation}\label{1.5}
 \left\{\begin{array}{l}
\bar{v}_{t}-\bar{u}_{x}=0,\\[2mm]
p(\bar{v})_{x}=- \alpha\bar{u},
 \end{array}
        \right.
\end{equation}
or
\begin{equation}\label{1.6}
 \left\{\begin{array}{l}
\bar{v}_{t}=-\frac{1}{\alpha}p(\bar{v})_{xx},\\[2mm]
p(\bar{v})_{x}=- \bar{u},
 \end{array}
        \right.
\end{equation}
with
\begin{equation}\label{1.7}
 (\bar{v},\bar{u})(x,t) \rightarrow (v_{\pm}, 0), \quad \mbox{as} \quad x \rightarrow \pm \infty.
\end{equation}
From $\eqref{1.1}_{1}$ and $\eqref{1.5}_{1}$, we have
\begin{equation}\label{1.8}
 (v-\bar{v})_{t}-(u-\bar{u})_{x}=0.
\end{equation}
Hinted by $\eqref{1.1}_{2}$, we suppose
\begin{equation}\label{1.9}
  u(x,t) \rightarrow {\rm e}^{-\alpha t}u_{\pm}  \quad \mbox{as} \quad x \rightarrow \pm \infty.
\end{equation}
Integrating \eqref{1.8} with respect to $x$, noting \eqref{1.9}, we obtain
\begin{equation}\label{1.10}
 \frac{{\rm d}}{{\rm d}t} \int_{-\infty}^{\infty} (v-\bar{v}) {\rm d}x ={\rm e}^{-\alpha t}(u_{+}-u_{-})=\frac{{\rm d}}{{\rm d}t}\left(\frac{u_{+}-u_{-}}{-\alpha}{\rm e}^{-\alpha t}  \cdot 1 \right),
\end{equation}
and hence
\begin{equation}\label{1.11}
 \frac{{\rm d}}{{\rm d}t} \int_{-\infty}^{\infty} \left[v(x,t)-\bar{v}(x+x_{0},t)-\frac{u_{+}-u_{-}}{-\alpha}{\rm e}^{-\alpha t}m_{0}(x)\right] {\rm d}x =0,
\end{equation}
 where $m_{0} \in C_{0}^{\infty}(\mathbb{R})$ satisfies
\begin{equation}\notag
  \int_{-\infty}^{\infty} m_{0}(x) {\rm d}x=1.
\end{equation}
Integrating \eqref{1.11} with respect to $t$, we obtain
\begin{equation}\label{1.12}
\begin{split}
 \int_{-\infty}^{\infty}& \left[v(x,t)-\bar{v}(x+x_{0},t)-\frac{u_{+}-u_{-}}{-\alpha}{\rm e}^{-\alpha t}m_{0}(x)\right] {\rm d}x\\
 &= \int_{-\infty}^{\infty} \left[v_{0}(x)-\bar{v}(x+x_{0},0)-\frac{u_{+}-u_{-}}{-\alpha}m_{0}(x) \right] {\rm d}x :=I(x_{0}).
 \end{split}
\end{equation}
Now, Let's determine $x_{0}$ such that $I(x_{0})=0$. Since
\begin{equation}\label{1.13}
\begin{split}
 I^{\prime}(x_{0})&=\frac{\partial}{\partial x_{0}}\left(\int_{-\infty}^{\infty}\left[v_{0}(x)-\bar{v}(x+x_{0},0)-\frac{u_{+}-u_{-}}{-\alpha}m_{0}(x)\right]\right){\rm d}x\\&
 =-\int_{-\infty}^{\infty}\bar{v}^{\prime}(x+x_{0},0){\rm d}x =-\left[\bar{v}(\infty,0)-\bar{v}(-\infty,0)\right]\\&=-(v_{+}-v_{-}),
\end{split}
\end{equation}
then we can obtain
\begin{equation}\label{1.14}
 I(x_{0})-I(0)=\int_{0}^{x_{0}}I^{\prime}(y){\rm d}y=-(v_{+}-v_{-})x_{0}.
\end{equation}
When $I(x_{0})=0$, we have
\begin{equation}\label{1.15}
 x_{0}=\frac{1}{v_{+}-v_{-}}I(0)=\frac{1}{v_{+}-v_{-}}\int_{-\infty}^{\infty} \left[v_{0}(x)-\bar{v}(x,0)-\frac{u_{+}-u_{-}}{-\alpha}m_{0}(x)\right] {\rm d}x.
\end{equation}
Thus, let's define
\begin{equation}\label{1.16}
  V(x,t)= \int_{-\infty}^{x} \left[v(y,t)-\bar{v}(y+x_{0},t)-\hat{v}(y,t) \right] {\rm d}y,
\end{equation}
with
\begin{equation}\label{1.17}
 \hat{v}(x,t)=\frac{u_{+}-u_{-}}{-\alpha}{\rm e}^{-\alpha t}m_{0}(x).
\end{equation}
Putting
\begin{equation}\label{1.18}
 \hat{u}(x,t)={\rm e}^{-\alpha t}\left[u_{-}+(u_{+}-u_{-})\int_{-\infty}^{x} m_{0}(y) {\rm d}y \right].
\end{equation}
Then one can immediately obtain
\begin{equation}\label{1.19}
 \left\{\begin{array}{l}
\hat{v}_{t}-\hat{u}_{x}=0,\\[2mm]
\hat{u}_{t}=- \alpha\hat{u}.
 \end{array}
        \right.
\end{equation}
Combining \eqref{1.1} and \eqref{1.5}, we get
\begin{equation}\label{1.20}
 \left\{\begin{array}{l}
(v-\bar{v}-\hat{v})_{t}-(u-\bar{u}-\hat{u})_{x}=0,\\[2mm]
(u-\bar{u}-\hat{u})_{t}+\bar{u}_{t}+ \left[p(v)-p(\bar{v})-g(u)f(v) \right]_{x}+\alpha(u-\bar{u}-\hat{u})=0.
 \end{array}
        \right.
\end{equation}
Setting
\begin{equation}\label{1.21}
  z(x,t)=u(x,t)-\bar{u}(x+x_{0},t)-\hat{u}(x,t),
\end{equation}
then from \eqref{1.16} and \eqref{1.21}, \eqref{1.20} can be transformed into
\begin{equation}\label{1.22}
 \left\{\begin{array}{l}
V_{t}-z=0,\\[2mm]
z_{t}+\left(p^\prime(\bar{v})V_{x}\right)_{x}+\alpha z=F_{1}+F_{2},\\[2mm]
(V,z)|_{t=0} :=(V_0,z_0)(x) \rightarrow 0 \quad \mbox{as} \quad x \rightarrow \pm \infty,
 \end{array}
        \right.
\end{equation}
or
\begin{equation}\label{1.23}
 \left\{\begin{array}{l}
V_{tt}+\left(p^\prime(\bar{v})V_{x}\right)_{x}+\alpha V_{t}=F_{1}+F_{2},\\[2mm]
(V,V_{t})|_{t=0} :=(V_0,z_0)(x) \rightarrow 0 \quad \mbox{as} \quad x \rightarrow \pm \infty,
 \end{array}
        \right.
\end{equation}
where
\begin{equation}\label{1.24}
 F_{1}:=\frac{1}{\alpha}p(\bar{v})_{xt}-\left(p(V_{x}+\bar{v}+\hat{v})-p(\bar{v})-p^\prime(\bar{v})V_{x} \right)_{x},
\end{equation}
\begin{equation}\label{1.25}
  F_{2}:= \left(g(z+\bar{u}+\hat{u})f(V_{x}+\bar{v}+\hat{v}) \right)_{x}.
\end{equation}
{\bf Notations.} In the following, $C$ and $c$ denote the generic positive constants depending only on the initial data,  but independent of the time. For any integer $m \geq 0$, we use $H^{m}$ to denote the usual Sobolev space $H^{m}\left(\mathbb{R}\right)$. Set $L^{2}=H^{m}$ when $m = 0$. For simplicity, the norm of $H^{m}$ is denoted by $\|\cdot\|_{m}$ with $\|\cdot\|_{0}=\|\cdot\|$.

In order to state our main result, we assume that the following assumptions hold:
\begin{align}
&\label{1.25a}\inf \limits_{|x|\leq |u_{\pm}|,{\min \left\{v_{+}, v_{-}\right\} \leq y \leq \max \left\{v_{+}, v_{-}\right\}}}\{g(x)f^\prime (y)-p^\prime (y)\}>0,\\
&\label{1.25b}p,f \in C^{3}(\mathbb{R}^{+}),~ p^\prime(v)<0~ \mbox{for}~ \mbox{any}~ v>0, \\
&\label{1.25c}g \in C^{3}(\mathbb{R}),~g(0)=g^\prime (0)=0.
\end{align}
The following are the main result.
\begin{theorem}\label{Thm1}
  Suppose that \eqref{1.25a}-\eqref{1.25c} hold, $\delta:=|v_{+}-v_{-}|+|u_{+}-u_{-}|$ and $\|V_{0}\|_{3}+\|z_{0}\|_{2}$ are sufficiency small, Then, there exists a unique time-global solution $(V, z)(x,t)$ of \eqref{1.22}, which satisfies
  \begin{equation}\notag
   V(x,t) \in C^{k}(0,\infty; H^{3-k}(\mathbb{R})), ~~ k=0,1,2,3, \quad z(x,t) \in C^{k}(0,\infty; H^{2-k}(\mathbb{R})), ~~ k=0,1,2,
  \end{equation}
  and
  \begin{align}
  &\label{1.26}\|\partial_{x}^{k}V(t)\|\leq C(1+t)^{-\frac{k}{2}},\qquad 0\leq k\leq 3,\\	
  &\label{1.27}\|\partial_{x}^{k}z(t)\|\leq C(1+t)^{-\frac{k}{2}-1},\qquad 0\leq k\leq 2,\\
  &\label{1.27a}(1+t)^{2}\|z_{t}(t)\|+(1+t)^{\frac{5}{2}}(\|z_{xt}(t)\|+\|z_{tt}(t)\|)\leq C.
  \end{align}
 Furthermore, under the additional assumption that $(V_{0}+\frac{1}{\alpha}z_{0})(x) \in L^{1}$, then the following improved decay estimates hold
 \begin{align}
&\label{1.28}\|\partial_{x}^{k}V(t)\| \leq C(1+t)^{-\frac{1}{4}-\frac{k}{2}},\quad 0\leq k \leq 3,\\
&\label{1.29}\|\partial_{x}^{k}z(t)\|\leq C(1+t)^{-\frac{1}{4}-\frac{k}{2}-1},\qquad 0\leq k\leq 2,\\
&\label{1.29a}(1+t)^{\frac{9}{4}}\|z_{t}(t)\|+(1+t)^{\frac{11}{4}}(\|z_{xt}(t)\|+\|z_{tt}(t)\|)\leq C.	
\end{align}
\end{theorem}
\begin{remark}\label{R1}
It should be noted that in Nishihara, Wang and Yang \cite{Nishihara-Wang-Yang2000}, the authors required that the initial perturbation $(V_{0}, z_{0})(x)$ be sufficiently small in $(H^{3} \cap L^{1}) \times (H^{2} \cap L^{1})$. But in our Theorem \ref{Thm1}, we require that $(V_{0}+\frac{1}{\alpha}z_{0})(x) \in L^{1}$ and the initial perturbation $(V_{0}, z_{0})(x)$ be sufficiently small in $H^{3} \times H^{2}$, which is weaker than those needed in \cite{Nishihara-Wang-Yang2000}.
\end{remark}
\begin{remark}\label{R2}
As we can see from \cite{Nishihara-Wang-Yang2000}, the authors obtained the optimal decay rates $\|\partial_{x}^{k}\partial_{t}^{l}V(t)\| \leq C(1+t)^{-\frac{1}{4}-\frac{k}{2}-l}$ for $0\leq k+l \leq 3$ and $0\leq l \leq 1$, however it was not clarified in the case of $l=2$. Actually, except for $\|z_{tt}(t)\|$, the convergence rates shown in our results Theorem \ref{Thm1} are all optimal. As for $\|z_{tt}(t)\|$, we can also use the similar way to obtain an extra time-decay $(1+t)^{-\frac{1}{2}}$ when $\left(V_{0},z_{0}\right)(x)$ is small belonging to $H^{4} \times H^{3}$. Thus in this sense, this results in this paper improves the decay rates obtained in \cite{Hsiao-Liu1992, Nishihara1996,Nishihara-Wang-Yang2000}.
\end{remark}
\begin{remark}\label{R3}
Compared with \cite{Hsiao-Liu1992, Nishihara1996, Nishihara-Wang-Yang2000}, in order to close the {\it a priori} assumption \eqref{3.1},
we require an additional technical condition \eqref{1.25a}. Notice that if $g(u)\equiv 0$ or $f(v)\equiv C$, the assumption \eqref{1.25a} is naturally true.
\end{remark}
\begin{remark}\label{R4}
 In fact, as for the case that $v_{+}=v_{-}>0$, the asymptotic profiles of the solutions are expected to be the constant states, we can still obtain the corresponding decay rates which are same as in Theorem \ref{Thm1}.
\end{remark}

The proof of existence and decay rates in Theorem \ref{Thm1} is based on the analysis of the nonlinear diffusion waves and classical energy estimates, as well as Green function method. In fact, compared with former arguments developed in \cite{Hsiao-Liu1992,Nishihara1996, Nishihara-Wang-Yang2000}, our conclusions can be regarded as a more general case, the main new ingredients in our analysis lie in the following.

Firstly, as usual, we can obtain the convergence rates of the solutions by the elementary energy estimates and some elaborate computations.
However, since the complexity of the expression for $F_{2}$, the energy estimates become much more complicate and more difficult because we have to face some extra difficult terms, such as $C{\rm e}^{-t}\|V(t)\|_{2}^{2}$ in \eqref{3.12}, $C{\rm e}^{-t}\|V_{xx}(t)\|^{2}$ in \eqref{3.19}, \eqref{3.25} and \eqref{3.30}, $\frac{1}{2}\frac{{\rm d}}{{\rm d}t}\int_{\mathbb{R}}gf^{\prime}V_{xx}^{2}{\rm d}x$ and in \eqref{3.30} and so on.
For the first two bad terms, actually, one can easy to see in Lemma \ref{L3.3} that $C$ is related to $|u_{\pm}|$, since we don't have the assumption that $|u_{\pm}|\ll 1$, it seems impossible to absorb them with some good terms, but notice that they all have the property of exponential decay, by employing the Gronwall's inequality, we succeeded in obtaining the desired estimate.
As for the last bad term, we require a technical condition \eqref{1.25a}, then it can be absorbed by $-\frac{1}{2}\frac{{\rm d}}{{\rm d}t}\int_{\mathbb{R}}p^\prime(\bar{v}) V_{xx}^{2}{\rm d}x$. One can see Section \ref{S3.1} for more detials. This is a new ingredient in this paper.

The second new ingredient in our analyses lies in the way to obtain the decay estimates \eqref{1.28}-\eqref{1.29a}. As we can see from the dissussions in \cite{Nishihara-Wang-Yang2000}, once they got the existence and decay rates of the solutions in the $L^{2}$-framework, by constructing an approximate Green function for the initial perturbations in $L^{1}$-sense, they obtained the improved decay rates $\|\partial_{x}^{k}\partial_{t}^{l}V(t)\| \leq C(1+t)^{-\frac{1}{4}-\frac{k}{2}-l}$ for $0\leq k+l \leq 3$ and $0\leq l \leq 1$. However, the calculation process is quite complicated and tedious, and the case of $l=2$ is not clarified.
In this paper, we employ a different strategy to derive the improved decay rates \eqref{1.28}-\eqref{1.29a}. Actually, after obtaining the existence and decay rate of the solution in the $L^{2}$-framework, we give the integral representation of the solution through Green function. Then by analyzing the integral representation of the solution, combined with the weighted energy estimate, we firstly obtain \eqref{1.28} for $0\leq k\leq 1$.
 With all these preparations, by continuing to use weighted energy estimates, we can obtain \eqref{1.28}-\eqref{1.29a}. See Section \ref{S3.2} for more detials.
This technique is quite useful and somewhat counterintuitive, it has been successfully used in \cite{Nishikawa1998, Geng-Zhang2015}.
 We think this approach has at least two advantages: one is that when $0\leq k+l \leq 3$ and $0\leq l \leq 2$, we can obtain the optimal decay rates on $\|\partial_{x}^{k}\partial_{t}^{l}V(t)\|$ without having to increase the regularity of the initial value, and the other is that the calculation process is much simpler and clearer. By the way, we will also use this approach to help us consider the asymptotic behavior of solutions to \eqref{1.1} on the quarter plane $\mathbb{R}^{+}\times \mathbb{R}^{+}$ in the future.

The last new ingredient in our analyses is reflected in the regularity requirement for the initial value. As we can see from \cite{Nishihara-Wang-Yang2000}, they obtained their main results under the condition that
\begin{equation}\notag
|v_{+}-v_{-}|+|u_{+}-u_{-}|+\|V_{0}\|_{3}+\|z_{0}\|_{2}+\|V_{0}\|_{L^{1}}+\|z_{0}\|_{L^{1}}	
\end{equation}
is sufficiently small. Compared with \cite{Nishihara-Wang-Yang2000}, the conditions in our result (Theorem \ref{Thm1}) are indeed much weaker.  The main reason is that we combine Green function theory with weighted energy estimates, which avoids us making complex higher-order estimates using only the integral representation of the solution.

The paper is organized as follows. In Section \ref{S2}, we prepare some preliminaries, which are useful in the proof of Theorem \ref{Thm1}. Section \ref{S3} is devoted to the proof of the convergence of the solutions $(v, u)(x,t)$ to the nonlinear diffusion waves.

\section{Preliminaries}\label{S2}

In this section, we are going to introduce some results on some fundamental properties of the nonlinear diffusion waves $(\bar{v},\bar{u})(x,t)$ and the correction functions $(\hat{v},\hat{u})(x,t)$, which will be used later.

From the previous works of van Duyn and Peletier in \cite{Duyn-Peletier1977}, we can know that the nonlinear diffusion equation $\eqref{1.6}_{1}$ and \eqref{1.7} have a unique self-similar solution called nonlinear diffusion wave in the form
\begin{equation}\label{2.1}
 \left\{\begin{array}{l}
\bar{v}(x, t)= \phi \left(\frac{x}{\sqrt{1+t}}\right):=\phi(\xi), \quad \xi \in \mathbb{R}, \\[3mm]
\phi(\pm \infty)=v_{\pm}.
\end{array}\right.
\end{equation}
Substituting $\eqref{2.1}_{1}$ into $\eqref{1.6}_{1}$, it follows that
\begin{equation}\label{2.2}
 \left(p^{\prime}(\phi(\xi))\phi^{\prime}(\xi)\right)^{\prime}=\frac{\alpha}{2} \xi \phi^{\prime}(\xi).
\end{equation}
Therefore, for any $\xi_{0} \in \mathbb{R}$, one has
\begin{equation}\label{2.3}
 \phi^{\prime}(\xi)=\frac{\phi^{\prime}\left(\xi_{0}\right)p^{\prime}(\phi(\xi_{0}))}{p^{\prime}(\phi(\xi))}{\rm e}^{\int_{\xi_{0}}^{\xi} \frac{\alpha\eta}{2p^{\prime}(\phi(\eta))}{\rm d}\eta}.
\end{equation}
As one can see in \cite{Hsiao-Liu1992}, it is easy to prove that the self-similar solution $\phi(\xi)$ satisfies
\begin{equation}\label{2.4}
  \sum_{k=1}^{4}\left|\frac{{\rm d}^{k}}{{\rm d} \xi^{k}} \phi(\xi)\right|+\left|\phi(\xi)-v_{+}\right|_{\{\xi>0\}}+\left|\phi(\xi)-v_{-}\right|_{\{\xi<0\}} \leq C\left|v_{+}-v_{-}\right| {\rm e}^{-c \xi^{2}},
\end{equation}
and $\bar{v}(x,t)$ satisfies the following dissipative properties:
\begin{equation}\label{2.5}
 \begin{split}
&\bar{v}_{x}=\frac{\phi^{\prime}(\xi)}{\sqrt{1+t}}, \quad \bar{v}_{t}=-\frac{\xi \phi^{\prime}(\xi)}{2(1+t)}, \quad \bar{v}_{xx}=\frac{\phi^{\prime \prime}(\xi)}{1+t},\quad \bar{v}_{xt}=-\frac{\phi^{\prime}(\xi)+\xi \phi^{\prime \prime}(\xi)}{2(1+t)^{\frac{3}{2}}},\quad \bar{v}_{x x x}=\frac{\phi^{\prime \prime \prime}(\xi)}{(1+t)^{\frac{3}{2}}}, \\[3mm]
&\bar{v}_{tt}=\frac{ \xi^{2}\phi^{\prime \prime}(\xi)+3\xi \phi^{\prime}(\xi)}{4(1+t)^{2}}, \quad \bar{v}_{xxt}=-\frac{\xi \phi^{\prime \prime \prime}(\xi)+2\phi^{\prime \prime}(\xi)}{2(1+t)^{2}},\quad \bar{v}_{xtt}=\frac{ \xi^{2}\phi^{\prime \prime \prime}(\xi)+3\phi^{\prime}(\xi)+5\xi \phi^{\prime \prime}(\xi)}{4(1+t)^{\frac{5}{2}}},\\[3mm]
&\bar{v}_{ttt}=-\frac{9\xi^{2}\phi^{\prime \prime}(\xi)+15\xi \phi^{\prime}(\xi)+\xi^{3}\phi^{\prime \prime \prime}(\xi)}{8(1+t)^{3}},\quad \bar{v}_{xxxx}=\frac{\phi^{(4)}(\xi)}{(1+t)^{2}},\quad \bar{v}_{xxxt}=-\frac{\xi \phi^{(4)}(\xi)+3\phi^{\prime \prime \prime}(\xi)}{2(1+t)^{\frac{5}{2}}},\\[3mm]
&\bar{v}_{xttt}=-\frac{12\xi^{2}\phi^{\prime \prime \prime}(\xi)+\xi^{3}\phi^{(4)}(\xi)+15\phi^{\prime}(\xi)+33\xi \phi^{\prime \prime}(\xi)}{8(1+t)^{\frac{7}{2}}},\quad \bar{v}_{xxtt}=\frac{8\phi^{\prime \prime}(\xi)+7\xi \phi^{\prime \prime \prime}(\xi)+\xi^{2} \phi^{(4)}(\xi)}{4(1+t)^{3}}.
\end{split}
\end{equation}
Combining \eqref{2.4} and \eqref{2.5}, we have the decay rates of the nonlinear diffusion waves $\bar{v}(x,t)$.
\begin{lemma}\label{L2.1}
For each $p\in [1,\infty]$ is an integer, the solution $\bar{v}(x,t)$ of \eqref{1.6}-\eqref{1.7} holds that
\begin{equation}\label{2.6}
 \begin{split}
 &\min \left\{v_{+}, v_{-}\right\} \leq \bar{v}(x,t) \leq \max \left\{v_{+}, v_{-}\right\}, \\[2mm]
 & \left\|\partial_{x}^{k} \partial_{t}^{j} \bar{v}(t)\right\|_{L^{p}} \leq C\left|v_{+}-v_{-}\right|(1+t)^{-\frac{k}{2}-j+\frac{1}{2p}}, \quad  k,j \geq 0,\quad 1 \leq k+j \leq 4.
\end{split}
\end{equation}
\end{lemma}
From \eqref{1.17} and \eqref{1.18}, one can immediately confirmed that the correction function $(\hat{v},\hat{u})(x,t)$ satisfies
\begin{lemma}\label{L2.2}
 Let $k, j$ be nonnegative integers and $p \in [1, \infty]$ is an integer, it holds that
 \begin{equation}\label{2.7}
  \begin{split}
    &\left\|\partial_{x}^{k} \partial_{t}^{j} \hat{v}(t)\right\|_{L^{p}}\leq C|u_{+}-u_{-}| {\rm e}^{-\alpha t}, \quad k \geq 0, ~~ j \geq 0,\\
    &\left\|\partial_{x}^{k} \partial_{t}^{j} \hat{u}(t)\right\|_{L^{p}}\leq C|u_{+}-u_{-}|{\rm e}^{-\alpha t}, \quad k \geq 1, ~~ j \geq 0,\\
    & \left\|\hat{u}(t)\right\|_{L^{\infty}}\leq \max\{|u_{+}|,|u_{-}|\}{\rm e}^{-\alpha t}.
    \end{split}
  \end{equation}
\end{lemma}
\begin{remark}\label{R2.1}
 It is easy to see that $\hat{u}(x,t)$ doesn't belong to any $L^{p}$ space for $1 \leq p < \infty$.
\end{remark}
Finally, we introduce the Sobolev inequation.
\begin{lemma}\label{L2.3}
 Let $f \in H^{1}(\mathbb{R})$, then
 \begin{equation}\label{2.8}
  \|f\|_{L^{\infty}} \leq \sqrt{2}\|f\|^{\frac{1}{2}}\|f_{x}\|^{\frac{1}{2}} .
  \end{equation}
\end{lemma}

\section{Proof of Theorem \ref{Thm1}}\label{S3}

In this section we devote ourselves to the proof of Theorem \ref{Thm1} concerning the existence, uniqueness and time decay rates of global
smooth solutions to \eqref{1.22}. In the first subsection, we shall prove the global existence, uniqueness and time decay rates by deriving the key uniform in-time {\it a priori} estimates in the $L^{2}$-framework. In the second subsection, we apply the Duhamel's principle combined with the weighted energy estimates to obtain the improved time decay rates of solutions. In what follows, we can put $\alpha=1$ without loss of generality, and denote $g(u)$ and $f(v)$ by $g$ and $f$ without any confusion.

\subsection{Proof of \eqref{1.26}-\eqref{1.27a}}\label{S3.1}

 The main purpose of this subsection is to study global existence and uniqueness of solutions to \eqref{1.22} in the $L^{2}$-framework, and obtain \eqref{1.26}-\eqref{1.27a}. It is well known that the global existence can be obtained by the continuation argument based on the local existence of solutions and {\it a priori} estimates. As for \eqref{1.22}, the local existence can be proved by the standard iteration method (cf. \cite{Kawashima1983, Nishida1978}) and its proof is omitted for brevity. In the following, we will devote ourselves to establish the following {\it a priori} estimates.

 \begin{proposition}\label{p1}
Assume that all the conditions in Theorem \ref{Thm1} hold, $V(x,t)$ is the smooth solution to the Cauchy problem \eqref{1.22} on $0\leq t \leq T$ for $T>0$. Then there exist constants $\varepsilon>0$ and $C>0$ shch that if
\begin{equation}\label{3.1}
\begin{split}
 N(T):=\sup \limits_{0 \leq t \leq T}&\left\{\sum \limits_{k=0}^{3}(1+t)^{k}\|\partial_{x}^{k}V(t)\|^{2}+ \sum \limits_{k=0}^{2}(1+t)^{k+2}\|\partial_{x}^{k}z(t)\|^{2} \right. \\
 &~~~~~·\left.+\sum \limits_{k=0}^{1}(1+t)^{k+4}\|\partial_{x}^{k}z_{t}(t)\|^{2}\right \} \leq \varepsilon^{2},
 \end{split}
\end{equation}
then it holds that
\begin{equation}\label{3.1a}
  \begin{split}
  \sum \limits_{k=0}^{3}&(1+t)^{k}\|\partial_{x}^{k}V(t)\|^{2}+ \sum \limits_{k=0}^{2}(1+t)^{k+2}\|\partial_{x}^{k}z(t)\|^{2}\\
  &+\int_{0}^{t}\left[\sum \limits_{j=1}^{3}(1+\tau)^{j-1}\|\partial_{x}^{j}V(\tau)\|^{2}+\sum \limits_{j=0}^{2}(1+\tau)^{j+1}\|\partial_{x}^{j}z(\tau)\|^{2}\right] {\rm d}\tau \\
  \leq &C(\|V_{0}\|^{2}_{3}+\|z_{0}\|^{2}_{2}+ \delta),
  \end{split}
  \end{equation}
  and
  \begin{equation}\label{3.1b}
  \begin{split}
  (1+t)^{4}&\|z_{t}(t)\|^{2}+(1+t)^{5}(\|z_{xt}(t)\|^{2}+\|z_{tt}(t)\|^{2})\\
  &+\int_{0}^{t}\left[(1+\tau)^{4}\|z_{xt}(\tau)\|^{2}+(1+\tau)^{5}\|z_{tt}(\tau)\|^{2}\right] {\rm d}\tau \\
  \leq& C(\|V_{0}\|^{2}_{3}+\|z_{0}\|^{2}_{2}+ \delta).
  \end{split}
  \end{equation}
\end{proposition}
 From \eqref{3.1} and the Sobolev inequality in Lemma \ref{L2.3}, one can immediately obtain
\begin{equation}\label{3.2}
\begin{split}
 &\|\partial_{x}^{k}V(t)\|_{L^{\infty}} \leq \sqrt{2}\varepsilon (1+t)^{-\frac{1}{4}-\frac{k}{2}},\quad k=0,1,2,\\
 &\|\partial_{x}^{k}z(t)\|_{L^{\infty}} \leq \sqrt{2}\varepsilon (1+t)^{-\frac{5}{4}-\frac{k}{2}},\quad k=0,1,\\
 &\|z_{t}(t)\|_{L^{\infty}} \leq \sqrt{2}\varepsilon (1+t)^{-\frac{9}{4}},
 \end{split}
\end{equation}
which will be used later. Then we shall prove the following lemma, which will play a key role in obtaining \eqref{3.1a}-\eqref{3.1b}.
\begin{lemma}\label{L3.1}
Assume that all the conditions in Proposition \ref{p1} hold, then it holds that
  \begin{align}\label{3.3}
&|(g^\prime f)(x,t)|\leq C(\varepsilon+\delta)(1+t)^{-\frac{1}{2}}+C{\rm e}^{-t},~~~|(g f^\prime)(x,t)|\leq C(\varepsilon+\delta)(1+t)^{-1}+C{\rm e}^{-t},\nonumber\\
&|(g^\prime f)_{x}(x,t)|\leq C(\varepsilon+\delta)(1+t)^{-1},~~~|(g^\prime f)_{t}(x,t)|\leq C(\varepsilon+\delta)(1+t)^{-\frac{3}{2}}+C{\rm e}^{-t},\nonumber\\
&|(g^\prime f)_{xx}(x,t)|\leq C|V_{xxt}(x,t)|+C(1+t)^{-\frac{1}{2}}|V_{xxx}(x,t)|+C(\varepsilon+\delta)(1+t)^{-\frac{3}{2}},\nonumber\\
&|(g^\prime f)_{xt}(x,t)|\leq C|V_{xtt}(x,t)|+C(1+t)^{-\frac{1}{2}}|V_{xxt}(x,t)|+C(\varepsilon+\delta)(1+t)^{-2},\nonumber\\
&|(g^\prime f)_{tt}(x,t)|\leq C|V_{ttt}(x,t)|+C(1+t)^{-\frac{1}{2}}|V_{xtt}(x,t)|+C(\varepsilon+\delta)(1+t)^{-\frac{5}{2}}+C{\rm e}^{-t},\nonumber\\
&|(g f^\prime)_{x}(x,t)|\leq C(\varepsilon+\delta)(1+t)^{-\frac{3}{2}},~~~|(g f^\prime)_{t}(x,t)|\leq C(\varepsilon+\delta)(1+t)^{-2}+C{\rm e}^{-t},\nonumber\\
&|(g f^\prime)_{xx}(x,t)|\leq C(1+t)^{-\frac{1}{2}}|V_{xxt}(x,t)|+C(1+t)^{-1}|V_{xxx}(x,t)|+C(\varepsilon+\delta)(1+t)^{-2},\nonumber\\
&|(g f^\prime)_{xt}(x,t)|\leq C(1+t)^{-\frac{1}{2}}|V_{xtt}(x,t)|+C(1+t)^{-1}|V_{xxt}(x,t)|+C(\varepsilon+\delta)(1+t)^{-\frac{5}{2}},\nonumber\\
&|(g f^\prime)_{tt}(x,t)|\leq C(1+t)^{-\frac{1}{2}}|V_{ttt}(x,t)|+C(1+t)^{-1}|V_{xtt}(x,t)|+C(\varepsilon+\delta)(1+t)^{-3}+C{\rm e}^{-t}.
\end{align}
\end{lemma}
\begin{proof}
By direct calculation, it follows from \eqref{1.25b}-\eqref{1.25c} and Taylor's expansion that
\begin{equation}\notag
  \begin{split}
    &\left|g^\prime f\right|\leq C|u|, ~~~~~ |(g f^\prime)|\leq C|u^{2}|, \\
    &\left|(g^\prime f)_{i}\right|\leq C(|u_{i}|+|uv_{i}|),~~~|(g f^\prime)_{i}|\leq C(|uu_{i}|+|u^{2}v_{i}|),\\
    &\left|\left(g^\prime f\right)_{ij}\right| \leq C(|u_{i}u_{j}|+|u_{ij}|+|u_{i}v_{j}|+|u_{j}v_{i}|+|uv_{i}v_{j}|+|uv_{ij}|),\\
    &\left|\left(g f^\prime\right)_{ij}\right| \leq C(|u_{i}u_{j}|+|uu_{ij}|+|uu_{i}v_{j}|+|uu_{j}v_{i}|+|u^{2}v_{i}v_{j}|+|u^{2}v_{ij}|),
    \end{split}
  \end{equation}
for $i,j=x ~\mbox{or}~ t$. Notice that $u=V_{t}+\bar{u}+\hat{u}$ and $v=V_{x}+\bar{v}+\hat{v}$, by using $\eqref{1.6}_{2}$, \eqref{2.6}-\eqref{2.8} and \eqref{3.2}, one can immediately obtain \eqref{3.3}.
\end{proof}
Now we turn to establish \eqref{3.1a}-\eqref{3.1b}, which will be given by series of lemmas.
\begin{lemma}\label{L3.2}
Under the assumptions of Proposition \ref{p1}, if $N(T) \leq \varepsilon^{2}$ and $\delta$ are small enough, it holds that
\begin{equation}\label{3.4}
   \|V(t)\|_{2}^{2}+\|V_{t}(t)\|_{1}^{2}+\int_{0}^{t}(\|V_{x}(\tau)\|_{1}^{2}+\|V_{t}(\tau)\|_{1}^{2}){\rm d}\tau \leq  C\left(\|V_{0}\|^{2}_{2}+\|z_{0}\|_{1}^{2}+ \delta \right),
  \end{equation}
  for $0 \leq t \leq T$.
\end{lemma}
\begin{proof}
 Firstly, multiplying $\eqref{1.23}_{1}$ by $V$ and integrating it with respect to $x$ over $\mathbb{R}$, we obtain
 \begin{align}\label{3.5}
   \frac{{\rm d}}{{\rm d}t}\int_{\mathbb{R}} \left(\frac{V^{2}}{2}+VV_{t}\right){\rm d} x-\int_{\mathbb{R}} p^\prime(\bar{v})V_{x}^{2}{\rm d}x =\int_{\mathbb{R}}V_{t}^{2}{\rm d}x+\int_{\mathbb{R}}F_{1}V{\rm d}x+\int_{\mathbb{R}}F_{2}V{\rm d} x .
 \end{align}
 In fact, the estimates of $\int_{\mathbb{R}}F_{1}V_{t}{\rm d} x$, $\int_{\mathbb{R}}F_{1}V{\rm d} x$, $\int_{\mathbb{R}}F_{1}V_{xx}{\rm d} x$, $\int_{\mathbb{R}}F_{1x}V_{xt}{\rm d} x$, $\int_{\mathbb{R}}F_{1x}V_{xxx}{\rm d} x$, $\int_{\mathbb{R}}F_{1xx}V_{xxt}{\rm d} x$, $\int_{\mathbb{R}}F_{1t}V_{tt}{\rm d} x$, $\int_{\mathbb{R}}F_{1t}V_{t}{\rm d} x$, $\int_{\mathbb{R}}F_{1xt}V_{xtt}{\rm d} x$, $\int_{\mathbb{R}}F_{1xt}V_{xt}{\rm d} x$ and $\int_{\mathbb{R}}F_{1tt}z_{tt}{\rm d} x$ have exactly shown in \cite{Nishihara1996}. For completeness, let's write
  \begin{equation}\label{3.6}
  \begin{split}
   \int_{\mathbb{R}} F_{1}V{\rm d} x &= \int_{\mathbb{R}}\left[-p^\prime(\bar{v})\bar{v}_{t}+p(V_{x}+\bar{v}+\hat{v})-p(\bar{v})-p^\prime(\bar{v})V_{x}\right]V_{x}{\rm d} x \\&
   \leq C(\varepsilon+\delta)\|V_{x}(t)\|^{2}+C\delta (1+t)^{-\frac{3}{2}}.
   \end{split}
  \end{equation}
  While, as for $F_{2}$ in \eqref{1.25},  by using \eqref{1.16} and \eqref{1.21}, we have
  \begin{equation}\label{3.7}
   F_{2}= \left(gf\right)_{x}=g^\prime f(V_{xt}-p(\bar{v})_{xx}+\hat{v}_{t})+gf^\prime(V_{xx}+\bar{v}_{x}+\hat{v}_{x}),
  \end{equation}
  then
  \begin{align}\label{3.8}
   \int_{\mathbb{R}}F_{2}V{\rm d}x=&\int_{\mathbb{R}}[g^\prime f(V_{xt}-p(\bar{v})_{xx}+\hat{v}_{t})+gf^\prime(V_{xx}+\bar{v}_{x}+\hat{v}_{x})]V{\rm d}x\nonumber\\
   =&\int_{\mathbb{R}}g^\prime fV_{xt}V{\rm d}x+\int_{\mathbb{R}}g^\prime f(-p(\bar{v})_{xx}+\hat{v}_{t})V{\rm d}x+\int_{\mathbb{R}}gf^\prime(V_{xx}+\bar{v}_{x}+\hat{v}_{x})V{\rm d}x\nonumber\\
   :=&I_{1}+I_{2}+I_{3}.
  \end{align}
  By using \eqref{1.25b}-\eqref{1.25c}, \eqref{2.6}-\eqref{3.2} and Taylor's expansion, we can derive that
  \begin{align}\label{3.9}
   I_{1}=&\int_{\mathbb{R}}g^\prime fV_{xt}V{\rm d}x\nonumber\\
   =&-\int_{\mathbb{R}}g^\prime fV_{t}V_{x}{\rm d}x-\int_{\mathbb{R}}VV_{t}[g^{\prime\prime} f(V_{xt}-p(\bar{v})_{xx}+\hat{v}_{t})+g^\prime f^\prime(V_{xx}+\bar{v}_{x}+\hat{v}_{x})]{\rm d}x\nonumber\\
   \leq&C\int_{\mathbb{R}}|V_{t}V_{x}|{\rm d}x+\frac{1}{2}\int_{\mathbb{R}}V_{t}^{2}(Vg^{\prime\prime}f)_{x}{\rm d}x+C\int_{\mathbb{R}}|VV_{t}|(\bar{v}_{x}^{2}+|\bar{v}_{xx}|+|\hat{v}_{t}|){\rm d}x \nonumber\\
   &+C\int_{\mathbb{R}}|VV_{t}|(|V_{t}|+|\bar{v}_{x}|+|\hat{u}|)(|V_{xx}|+|\bar{v}_{x}|+|\hat{v}_{x}|){\rm d}x \nonumber\\
   \leq& -\frac{p^\prime(\bar{v})}{32}\int_{\mathbb{R}}V_{x}^{2}{\rm d}x+C\int_{\mathbb{R}}V_{t}^{2}{\rm d}x+C\|V(t)\|_{L^{\infty}}^{2}\int_{\mathbb{R}}(\bar{v}_{x}^{4}+\bar{v}_{xx}^{2}+\hat{v}_{t}^{2}){\rm d}x\nonumber\\
   &+C\|VV_{xx}(t)\|_{L^{\infty}}^{2}\int_{\mathbb{R}}\bar{v}_{x}^{2}{\rm d}x+C\int_{0}^{\infty}\hat{v}_{x}^{2}{\rm d}x+\|\hat{u}(t)\|_{L^{\infty}}^{2}\int_{\mathbb{R}}(V_{xx}^{2}+\bar{v}_{x}^{2}){\rm d}x\nonumber\\
   \leq& -\frac{p^\prime(\bar{v})}{16}\|V_{x}(t)\|^{2} +C\|V_{t}(t)\|^{2}+C{\rm e}^{-t}\|V_{xx}(t)\|^{2}+ C\delta (1+t)^{-2},
  \end{align}
  \begin{align}\label{3.10}
   I_{2}=&\int_{\mathbb{R}}g^\prime f(-p(\bar{v})_{xx}+\hat{v}_{t})V{\rm d}x\nonumber\\
   \leq&C\int_{\mathbb{R}}(|V_{t}|+|\bar{v}_{x}|+|\hat{u}|)(\bar{v}_{x}^{2}+|\bar{v}_{xx}|+|\hat{v}_{t}|)|V|{\rm d}x\nonumber\\
   \leq&\|V_{t}(t)\|^{2}+C\delta (1+t)^{-2}\|V(t)\|^{2}+C\delta (1+t)^{-1}\|\bar{v}_{x}(t)\|\|V(t)\|\nonumber\\
   &+C\|\hat{u}(t)\|_{L^{\infty}}(\|\bar{v}_{x}(t)\|^{2}+\|\bar{v}_{xx}(t)\|_{L^{1}}+\|\hat{v}_{t}(t)\|_{L^{1}})\nonumber\\
   \leq&\|V_{t}(t)\|^{2}+C\delta (1+t)^{-\frac{5}{4}},
  \end{align}
  and
  \begin{align}\label{3.11}
    I_{3}=&\int_{\mathbb{R}}gf^\prime(V_{xx}+\bar{v}_{x}+\hat{v}_{x})V{\rm d}x\nonumber\\
    \leq&C\int_{\mathbb{R}}(V_{t}-p(\bar{v})_{x}+\hat{u})^{2}(|V_{xx}|+|\bar{v}_{x}|+|\hat{v}_{x}|)|V|{\rm d}x\nonumber\\
    \leq&C\|V_{t}(t)\|^{2}+C\int_{\mathbb{R}}[|V_{t}|(|\bar{v}_{x}|+|\hat{u}|)+|\bar{v}_{x}|^{2}+|\bar{v}_{x}||\hat{u}|](|V_{xx}|+|\bar{v}_{x}|+|\hat{v}_{x}|)|V|{\rm d} x \nonumber\\
  &+C\|\hat{u}(t)\|_{L^{\infty}}^{2}\int_{\mathbb{R}}|V_{xx}V|{\rm d} x+C\|\hat{u}(t)\|_{L^{\infty}}^{2}\int_{\mathbb{R}}|V|(|\bar{v}_{x}|+|\hat{v}_{x}|){\rm d} x\nonumber\\
  \leq& C\|V_{t}(t)\|^{2}+C{\rm e}^{-t}\|V(t)\|_{2}^{2}+ C\delta (1+t)^{-2}+C\delta (1+t)^{-1}\|\bar{v}_{x}(t)\|\|V(t)\|\nonumber\\
  \leq& C\|V_{t}(t)\|^{2}+C{\rm e}^{-t}\|V(t)\|_{2}^{2}+ C\delta (1+t)^{-\frac{5}{4}}.
  \end{align}
  Substituting \eqref{3.6} and \eqref{3.8}-\eqref{3.11} into \eqref{3.5}, we obtain
\begin{equation}\label{3.12}
 \frac{{\rm d}}{{\rm d}t}\int_{\mathbb{R}} \left(\frac{V^{2}}{2}+VV_{t}\right){\rm d} x-\frac{3}{4}\int_{\mathbb{R}} p^\prime(\bar{v})V_{x}^{2}{\rm d}x \leq C\|V_{t}(t)\|^{2}+C{\rm e}^{-t}\|V(t)\|_{2}^{2}+ C\delta (1+t)^{-\frac{5}{4}}.
\end{equation}
Next, multiplying $\eqref{1.23}_{1}$ by $V_{t}$ and integrating it with respect to $x$ over $\mathbb{R}$, after some integrations by parts, we get
\begin{equation}\label{3.13}
  \frac{1}{2}\frac{{\rm d}}{{\rm d}t}\int_{\mathbb{R}} \left(V_{t}^{2}-p^\prime(\bar{v})V_{x}^{2}\right){\rm d} x+\int_{\mathbb{R}} V_{t}^{2}{\rm d} x=-\frac{1}{2}\int_{\mathbb{R}}p^{\prime\prime}(\bar{v})\bar{v}_{t}V_{x}^{2}{\rm d} x+\int_{\mathbb{R}}F_{1}V_{t}{\rm d}x+\int_{\mathbb{R}}F_{2}V_{t}{\rm d}x .
  \end{equation}
  We now estimate the righthand side of \eqref{3.13} term by term. Firstly, we have from \eqref{2.6} that
   \begin{equation}\label{3.14}
   -\frac{1}{2}\int_{\mathbb{R}}p^{\prime\prime}(\bar{v})\bar{v}_{t}V_{x}^{2}{\rm d} x \leq C\delta (1+t)^{-1}\|V_{x}(t)\|^{2}.
   \end{equation}
  As exactly shown in \cite{Nishihara1996}, we have
  \begin{equation}\label{3.15}
   \begin{split}
    \int_{\mathbb{R}} F_{1}V_{t}{\rm d}x&=\int_{\mathbb{R}} p(\bar{v})_{xt}V_{t}{\rm d} x+\frac{{\rm d}}{{\rm d}t}\int_{\mathbb{R}}\left[\int_{\bar{v}}^{V_{x}+\bar{v}+\hat{v}}p(s){\rm d}s-p(\bar{v})V_{x}-\frac{p^\prime(\bar{v})}{2}V_{x}^{2}\right]{\rm d}x \\&
    ~~~~+\int_{\mathbb{R}}\left[-p(V_{x}+\bar{v}+\hat{v})+p(\bar{v})+p^\prime(\bar{v})V_{x}+\frac{p^{\prime\prime}(\bar{v})}{2}V_{x}^{2}\right]\bar{v}_{t}{\rm d} x-\int_{\mathbb{R}}p(V_{x}+\bar{v}+\hat{v})\hat{v}_{t}{\rm d} x\\&
    \leq \frac{1}{16}\|V_{t}(t)\|^{2}+C(\varepsilon +\delta) (1+t)^{-1} \|V_{x}(t)\|^{2}+C\delta (1+t)^{-\frac{5}{2}}\\&
    ~~~~+\frac{{\rm d}}{{\rm d}t}\int_{\mathbb{R}}\left[\int_{\bar{v}}^{V_{x}+\bar{v}+\hat{v}}p(s){\rm d}s-p(\bar{v})V_{x}-\frac{p^\prime(\bar{v})}{2}V_{x}^{2}\right]{\rm d}x.
    \end{split}
  \end{equation}
  Now we deal with the last term of the righthand side of \eqref{3.13}. Notice that
  \begin{align}\label{3.16}
   \int_{\mathbb{R}}F_{2}V_{t}{\rm d}x=&\int_{\mathbb{R}}g^\prime f(V_{xt}-p(\bar{v})_{xx}+\hat{v}_{t})V_{t}{\rm d}x+\int_{\mathbb{R}}gf^\prime(V_{xx}+\bar{v}_{x}+\hat{v}_{x})V_{t}{\rm d}x\nonumber\\
   :=&I_{4}+I_{5}.
  \end{align}
  From \eqref{1.25b}-\eqref{1.25c} and \eqref{2.6}-\eqref{3.2} and Taylor's expansion, we can derive
  \begin{align}\label{3.17}
   I_{4}=&\int_{\mathbb{R}}g^\prime f(V_{xt}-p(\bar{v})_{xx}+\hat{v}_{t})V_{t}{\rm d}x\nonumber\\
   \leq&\int_{\mathbb{R}}g^\prime fV_{xt}V_{t}{\rm d}x+C\int_{\mathbb{R}}(|V_{t}|+|\bar{v}_{x}|+|\hat{u}|)(\bar{v}_{x}^{2}+|\bar{v}_{xx}|+|\hat{v}_{t}|)|V_{t}|{\rm d}x\nonumber\\
   \leq&-\frac{1}{2}\int_{\mathbb{R}}(g^\prime f)_{x}V_{t}^{2}{\rm d}x+\frac{1}{32}\int_{\mathbb{R}}V_{t}^{2}{\rm d}x+C(1+t)^{-1}\int_{\mathbb{R}}(\bar{v}_{x}^{4}+\bar{v}_{xx}^{2}+\hat{v}_{t}^{2}){\rm d}x\nonumber\\
   \leq&\frac{1}{16}\|V_{t}(t)\|^{2}+C\delta (1+t)^{-\frac{5}{2}},
  \end{align}
  and
  \begin{align}\label{3.18}
    I_{5}=&\int_{\mathbb{R}}gf^\prime(V_{xx}+\bar{v}_{x}+\hat{v}_{x})V_{t}{\rm d}x\nonumber\\
    \leq&C\int_{\mathbb{R}}(V_{t}-p(\bar{v})_{x}+\hat{u})^{2}(|V_{xx}|+|\bar{v}_{x}|+|\hat{v}_{x}|)|V_{t}|{\rm d}x\nonumber\\
    \leq&C(\varepsilon +\delta)\int_{\mathbb{R}}V_{t}^{2}{\rm d}x+C\int_{\mathbb{R}}(|\bar{v}_{x}|^{2}+|\bar{v}_{x}||\hat{u}|)(|V_{xx}|+|\bar{v}_{x}|+|\hat{v}_{x}|)|V_{t}|{\rm d} x \nonumber\\
  &+C\|\hat{u}(t)\|_{L^{\infty}}^{2}\int_{\mathbb{R}}|V_{xx}V_{t}|{\rm d} x+C\|\hat{u}(t)\|_{L^{\infty}}^{2}\int_{\mathbb{R}}|V_{t}|(|\bar{v}_{x}|+|\hat{v}_{x}|){\rm d} x\nonumber\\
  \leq&\frac{1}{16}\|V_{t}(t)\|^{2}+C{\rm e}^{-t}\|V_{xx}(t)\|^{2}+C\delta (1+t)^{-\frac{5}{2}}.
  \end{align}
  Substituting \eqref{3.14}-\eqref{3.18} into \eqref{3.13}, we have
  \begin{align}\label{3.19}
   &\frac{1}{2}\frac{{\rm d}}{{\rm d}t}\int_{\mathbb{R}} \left(V_{t}^{2}-p^\prime(\bar{v})V_{x}^{2}\right){\rm d} x+\frac{3}{4}\int_{\mathbb{R}} V_{t}^{2}{\rm d}x\nonumber\\
   \leq& C(\varepsilon +\delta) (1+t)^{-1} \|V_{x}(t)\|^{2}+C\delta (1+t)^{-\frac{5}{2}}+C{\rm e}^{-t}\|V_{xx}(t)\|^{2}\nonumber\\
   &+\frac{{\rm d}}{{\rm d}t}\int_{\mathbb{R}}\left[\int_{\bar{v}}^{V_{x}+\bar{v}+\hat{v}}p(s){\rm d}s-p(\bar{v})V_{x}-\frac{p^\prime(\bar{v})}{2}V_{x}^{2}\right]{\rm d}x.
  \end{align}
  Addition of $\lambda \cdot \eqref{3.12}$, $0<\lambda \ll 1$ to \eqref{3.19} yields
  \begin{align}\label{3.20}
    &\frac{1}{2}\frac{{\rm d}}{{\rm d}t}\int_{\mathbb{R}} \left(V_{t}^{2}+\lambda V^{2}+2\lambda VV_{t}-p^\prime(\bar{v})V_{x}^{2}\right){\rm d} x+\frac{1}{2}\int_{\mathbb{R}}(V_{t}^{2}-\lambda p^\prime(\bar{v})V_{x}^{2}){\rm d}x\nonumber\\
     \leq& C{\rm e}^{-t}\|V(t)\|_{2}^{2}+ C\delta (1+t)^{-\frac{5}{4}}+\frac{{\rm d}}{{\rm d}t}\int_{\mathbb{R}}\left[\int_{\bar{v}}^{V_{x}+\bar{v}+\hat{v}}p(s){\rm d}s-p(\bar{v})V_{x}-\frac{p^\prime(\bar{v})}{2}V_{x}^{2}\right]{\rm d}x.
  \end{align}
  Integrating \eqref{3.20} respect to $t$ over $[0,t]$, we have
  \begin{equation}\label{3.21}
   \begin{split}
    \|V(t)\|_{1}^{2}&+\|V_{t}(t)\|^{2}+\int_{0}^{t}(\|V_{x}(\tau)\|^{2}+\|V_{t}(\tau)\|^{2}){\rm d}\tau \\&
     \leq  C\left(\|V_{0}\|^{2}_{1}+\|z_{0}\|^{2}+ \delta \right)+C\int_{0}^{t}{\rm e}^{-\tau}\|V(\tau)\|_{2}^{2}{\rm d}\tau .
    \end{split}
  \end{equation}
  Now we consider the higher order energy estimates. Multiplying $\eqref{1.23}_{1}$ by $-V_{xx}$ and integrating it with respect to $x$ over $\mathbb{R}$, we obtain
  \begin{equation}\label{3.22}
  \begin{split}
  \frac{1}{2}\frac{{\rm d}}{{\rm d}t}\int_{\mathbb{R}} \left(V_{x}^{2}+2V_{x}V_{xt}\right){\rm d} x-\int_{\mathbb{R}}p^\prime(\bar{v})V_{xx}^{2} {\rm d}x=&\int_{\mathbb{R}}V_{xt}^{2}{\rm d} x+\int_{\mathbb{R}}(p^\prime(\bar{v})_{x}V_{x}V_{xx}-F_{1}V_{xx}){\rm d}x\\
  &-\int_{\mathbb{R}}F_{2}V_{xx}{\rm d}x.
  \end{split}
  \end{equation}
  We estimate the right hand side of \eqref{3.22} as follows. Firstly, it is easy to see that
  \begin{equation}\label{3.23}
   \int_{\mathbb{R}}(p^\prime(\bar{v})_{x}V_{x}V_{xx}-F_{1}V_{xx}){\rm d}x\leq -\frac{p^\prime(\bar{v})}{16}\|V_{xx}(t)\|^{2}+C\delta(1+t)^{-1}\|V_{x}(t)\|^{2}+C\delta (1+t)^{-\frac{5}{2}},
  \end{equation}
  then from \eqref{1.25b}-\eqref{1.25c} and \eqref{2.6}-\eqref{3.3}, we have
  \begin{align}\label{3.24}
  &-\int_{\mathbb{R}}F_{2}V_{xx}{\rm d}x\nonumber\\
  =&-\int_{\mathbb{R}}g^\prime f(V_{xt}-p(\bar{v})_{xx}+\hat{v}_{t})V_{xx}{\rm d}x-\int_{\mathbb{R}}gf^\prime(V_{xx}+\bar{v}_{x}+\hat{v}_{x})V_{xx}{\rm d}x\nonumber\\
  \leq& C\int_{\mathbb{R}}|V_{xt}||V_{xx}|{\rm d}x+C(1+t)^{-\frac{1}{2}}\int_{\mathbb{R}}(|\bar{v}_{xx}|+|\bar{v}_{x}|^{2}+|\hat{v}_{t}|)|V_{xx}|{\rm d}x \nonumber\\
  &+C(\varepsilon +\delta)\int_{\mathbb{R}}V_{xx}^{2}{\rm d}x+C{\rm e}^{-t}\int_{\mathbb{R}}V_{xx}^{2}{\rm d}x+C(1+t)^{-1}\int_{\mathbb{R}}(|\bar{v}_{x}|+|\hat{v}_{x}|)|V_{xx}|{\rm d}x\nonumber\\
  \leq& -\frac{p^\prime(\bar{v})}{16}\|V_{xx}(t)\|^{2}+C\|V_{xt}(t)\|^{2}+C{\rm e}^{-t}\|V_{xx}(t)\|^{2}+C\delta (1+t)^{-\frac{5}{2}}.
  \end{align}
  Substituting \eqref{3.23}-\eqref{3.24} into \eqref{3.22}, we have
  \begin{equation}\label{3.25}
  \begin{split}
     &\frac{1}{2}\frac{{\rm d}}{{\rm d}t}\int_{\mathbb{R}}(V_{x}^{2}+2V_{x}V_{xt}){\rm d}x-\frac{3}{4}\int_{\mathbb{R}}p^\prime(\bar{v})V_{xx}^{2} {\rm d}x \\
     \leq& C\|V_{xt}(t)\|^{2}+C{\rm e}^{-t}\|V_{xx}(t)\|^{2}+C\delta(1+t)^{-1}\|V_{x}(t)\|^{2}+C\delta (1+t)^{-\frac{5}{2}}.
     \end{split}
  \end{equation}
  Then the calculations of $\int_{\mathbb{R}}\eqref{1.23}_{1x}\times V_{xt}{\rm d}x$ gives
  \begin{equation}\label{3.26}
  \begin{split}
   \frac{1}{2}\frac{{\rm d}}{{\rm d}t}\int_{\mathbb{R}}\left(V_{xt}^{2}-p^\prime(\bar{v})V_{xx}^{2}\right){\rm d}x+\int_{\mathbb{R}}V_{xt}^{2}{\rm d}x=&-\frac{1}{2}\int_{\mathbb{R}}p^{\prime\prime}(\bar{v})\bar{v}_{t}V_{xx}^{2}{\rm d}x+\int_{\mathbb{R}}\left(F_{1}-p^\prime(\bar{v})_{x}V_{x}\right)_{x}V_{xt}{\rm d}x\\
   &+\int_{\mathbb{R}}F_{2x}V_{xt}{\rm d}x.
   \end{split}
  \end{equation}
  We estimate the right hand side of \eqref{3.26} as follows. Firstly, by applying \eqref{2.6}, one gets
 \begin{equation}\label{3.27}
 -\frac{1}{2}\int_{\mathbb{R}}p^{\prime\prime}(\bar{v})\bar{v}_{t}V_{xx}^{2}{\rm d}x \leq C\delta(1+t)^{-1}\|V_{xx}(t)\|^{2}.	
 \end{equation}
  Next from \eqref{2.6}-\eqref{2.8} and {\it a priori} assumption \eqref{3.2}, we get
  \begin{equation}\label{3.28}
   \begin{split}
    &\int_{\mathbb{R}}\left(F_{1}-p^\prime(\bar{v})_{x}V_{x}\right)_{x}V_{xt}{\rm d}x \\
     \leq& \frac{1}{16} \|V_{xt}(t)\|^{2}+\frac{1}{2}\frac{{\rm d}}{{\rm d}t}\int_{\mathbb{R}}\left[p^\prime(V_{x}+\bar{v}+\hat{v})-p^\prime(\bar{v})\right]V_{xx}^{2}{\rm d}x+C\delta (1+t)^{-\frac{7}{2}}\\
    &+C(\varepsilon +\delta)(1+t)^{-1}\|V_{xx}(t)\|^{2}+C\delta(1+t)^{-2}\|V_{x}(t)\|^{2}.
    \end{split}
  \end{equation}
  Now we estimate the last term in the right hand of \eqref{3.26}. From \eqref{2.6}-\eqref{3.3}, we can derive that
  \begin{align}\label{3.29}
    &\int_{\mathbb{R}}F_{2x}V_{xt}{\rm d}x\nonumber\\
    =&\int_{\mathbb{R}}(g^\prime fV_{xt}+gf^\prime V_{xx})_{x}V_{xt}{\rm d}x+\int_{\mathbb{R}}[g^\prime f(-p(\bar{v})_{xx}+\hat{v}_{t})]_{x}V_{xt}{\rm d}x+\int_{\mathbb{R}}[gf^\prime(\bar{v}_{x}+\hat{v}_{x})]_{x}V_{xt}{\rm d}x\nonumber\\
    =&-\int_{\mathbb{R}}(g^\prime fV_{xt}+gf^\prime V_{xx})V_{xxt}{\rm d}x+C(1+t)^{-\frac{1}{2}}\int_{\mathbb{R}}(|\bar{v}_{xxx}|+|\bar{v}_{x}||\bar{v}_{xx}|+|\bar{v}_{x}|^{3}+|\hat{v}_{xt}|)|V_{xt}|{\rm d}x \nonumber\\
    &+C(1+t)^{-1}\int_{\mathbb{R}}(|\bar{v}_{xx}|+|\bar{v}_{x}|^{2}+|\hat{v}_{t}|+|\hat{v}_{xx}|)|V_{xt}|{\rm d}x+C(1+t)^{-\frac{3}{2}}\int_{\mathbb{R}}(|\bar{v}_{x}|+|\hat{v}_{x}|)|V_{xt}|{\rm d}x \nonumber\\
   \leq&\frac{1}{2}\int_{\mathbb{R}}\left(g^\prime f\right)_{x}V_{xt}^{2}{\rm d}x-\frac{1}{2}\frac{{\rm d}}{{\rm d}t}\int_{\mathbb{R}}gf^\prime V_{xx}^{2}{\rm d}x+\frac{1}{2}\int_{\mathbb{R}}\left(gf^\prime\right)_{t}V_{xx}^{2}{\rm d}x+\frac{1}{32}\int_{\mathbb{R}}V_{xt}^{2}{\rm d}x+C\delta (1+t)^{-\frac{7}{2}} \nonumber\\
     \leq& \frac{1}{16}\|V_{xt}(t)\|^{2}+ C(\varepsilon +\delta)(1+t)^{-2}\|V_{xx}(t)\|^{2}-\frac{1}{2}\frac{{\rm d}}{{\rm d}t}\int_{\mathbb{R}}gf^\prime V_{xx}^{2}{\rm d}x\nonumber\\
    &+C{\rm e}^{-t}\|V_{xx}(t)\|^{2}+C\delta (1+t)^{-\frac{7}{2}}.
  \end{align}
  Substituting \eqref{3.27}-\eqref{3.29} into \eqref{3.26}, we have
  \begin{align}\label{3.30}
    &\frac{1}{2}\frac{{\rm d}}{{\rm d}t}\int_{\mathbb{R}}\left[V_{xt}^{2}+\left(gf^\prime-p^\prime(\bar{v})\right)V_{xx}^{2}\right]{\rm d}x+\frac{3}{4}\int_{\mathbb{R}}V_{xt}^{2}{\rm d}x \nonumber\\
    \leq& \frac{1}{2}\frac{{\rm d}}{{\rm d}t}\int_{\mathbb{R}}\left[p^\prime(V_{x}+\bar{v}+\hat{v})-p^\prime(\bar{v})\right]V_{xx}^{2}{\rm d}x+C(\varepsilon +\delta)(1+t)^{-1}\|V_{xx}(t)\|^{2}\nonumber\\
    &+C\delta(1+t)^{-2}\|V_{x}(t)\|^{2}+C{\rm e}^{-t}\|V_{xx}(t)\|^{2}+C\delta (1+t)^{-\frac{7}{2}}.
  \end{align}
  Addition of $\lambda \cdot \eqref{3.25}$ to \eqref{3.30} $(0<\lambda \ll 1)$, one has that
  \begin{equation}\label{3.31}
   \begin{split}
   &\frac{1}{2}\frac{{\rm d}}{{\rm d}t}\int_{\mathbb{R}} \left[V_{xt}^{2}+\lambda V_{x}^{2}+2\lambda V_{xt}V_{x}+\left(gf^\prime-p^\prime(\bar{v})\right)V_{xx}^{2}\right]{\rm d}x+\frac{1}{2}\int_{\mathbb{R}} \left(V_{xt}^{2}-\lambda p^\prime(\bar{v})V_{xx}^{2}\right){\rm d} x \\
   \leq& \frac{1}{2}\frac{{\rm d}}{{\rm d}t}\int_{\mathbb{R}}\left[p^\prime(V_{x}+\bar{v}+\hat{v})-p^\prime(\bar{v})\right]V_{xx}^{2}{\rm d}x+ C\delta (1+t)^{-\frac{5}{2}}+C\delta(1+t)^{-1}\|V_{x}(t)\|^{2}\\&
   ~~~~+C{\rm e}^{-t}\|V_{xx}(t)\|^{2} .
    \end{split}
  \end{equation}
   Since $\varepsilon +\delta \ll 1$, by using \eqref{1.25a}, \eqref{2.6}-\eqref{3.2}, it is easy to obtain that
   \begin{equation}\label{3.32}
   	gf^\prime-p^\prime(\bar{v})\geq c_{0}>0,
   \end{equation}
   where $c_{0}$ is only depend on $|u_{\pm}|$ and $|v_{\pm}|$.
   Integrating \eqref{3.31} over $[0,t]$ and using \eqref{3.32}, we have
  \begin{equation}\label{3.33}
   \begin{split}
   \|V_{x}(t)\|&_{1}^{2}+\|V_{xt}(t)\|^{2}+\int_{0}^{t}(\|V_{xx}(\tau)\|^{2}+\|V_{xt}(\tau)\|^{2}){\rm d}\tau \\&
   \leq  C\left(\|V_{0}\|^{2}_{2}+\|z_{0}\|_{1}^{2}+ \delta \right)+C\int_{0}^{t}{\rm e}^{-\tau}\|V_{xx}(\tau)\|^{2}{\rm d}\tau + C\delta \int_{0}^{t}\|V_{x}(\tau)\|^{2}{\rm d}\tau .
    \end{split}
  \end{equation}
  It follows from \eqref{3.21} and \eqref{3.33} that
  \begin{equation}\label{3.34}
   \begin{split}
   \|V(t)\|_{2}^{2}+&\|V_{t}(t)\|_{1}^{2}+\int_{0}^{t}(\|V_{x}(\tau)\|_{1}^{2}+\|V_{t}(\tau)\|_{1}^{2}){\rm d}\tau \\&
   \leq  C\left(\|V_{0}\|^{2}_{2}+\|z_{0}\|_{1}^{2}+ \delta \right)+C\int_{0}^{t}{\rm e}^{-\tau}\|V(\tau)\|_{2}^{2}{\rm d}\tau ,
    \end{split}
  \end{equation}
  which implies by Gronwall's inequality that
  \begin{equation}\notag
   \sup \limits_{0 \leq t \leq T}\{\|V(t)\|_{2}^{2}+\|V_{t}(t)\|_{1}^{2}\} \leq C\left(\|V_{0}\|^{2}_{2}+\|z_{0}\|_{1}^{2}+ \delta \right).
  \end{equation}
  Then combining the above two equations, one can obtain \eqref{3.4}. The proof of Lemma 3.2 is completed.
\end{proof}
\begin{lemma}\label{L3.3}
Under the assumptions of Proposition \ref{p1}, if $N(T) \leq \varepsilon^{2}$ and $\delta$ are small enough, it holds that
\begin{equation}\label{3.35}
   (1+t)(\|V_{x}(t)\|^{2}+\|V_{t}(t)\|^{2})+\int_{0}^{t}(1+\tau)\|V_{t}(\tau)\|^{2}{\rm d}\tau \leq  C\left(\|V_{0}\|^{2}_{2}+\|z_{0}\|_{1}^{2}+ \delta \right),
  \end{equation}
 \begin{equation}\label{3.36}
  \begin{split}
   (1+t)^{2}&(\|V_{xx}(t)\|^{2}+\|V_{xt}(t)\|^{2})+\int_{0}^{t}\left[(1+\tau)\|V_{xx}(\tau)\|^{2}+(1+\tau)^{2}\|V_{xt}(\tau)\|^{2}\right]{\rm d}\tau \\&
   \leq  C\left(\|V_{0}\|^{2}_{2}+\|z_{0}\|_{1}^{2}+ \delta \right),
   \end{split}
  \end{equation}
  for $0 \leq t \leq T$.
\end{lemma}
\begin{proof}
 Multiplying \eqref{3.19} by $(1+t)$ and integrating by parts, we have
 \begin{align}\label{3.37}
   &\frac{1}{2}\frac{{\rm d}}{{\rm d}t}\int_{\mathbb{R}}(1+t)\left(V_{t}^{2}-p^\prime(\bar{v})V_{x}^{2}\right){\rm d} x+\frac{3}{4}\int_{\mathbb{R}}(1+t)V_{t}^{2}{\rm d}x\nonumber\\
   \leq& \frac{{\rm d}}{{\rm d}t}(1+t)\int_{\mathbb{R}}\left[\int_{\bar{v}}^{V_{x}+\bar{v}+\hat{v}}p(s){\rm d}s-p(\bar{v})V_{x}-\frac{p^\prime(\bar{v})}{2}V_{x}^{2}\right]{\rm d}x+C(\|V_{x}(t)\|_{1}^{2}+\|V_{t}(t)\|^{2})\nonumber\\
   &+C\delta (1+t)^{-\frac{3}{2}}.
  \end{align}
Integrating the above inequality in $t$ over $[0, t]$ and using \eqref{3.4}, we can immediately obtain \eqref{3.35}.
Then multiplying \eqref{3.31} by $(1+t)$ and integrating it with respect to $t$, we obtain that
\begin{equation}\label{3.38}
 \begin{split}
   (1+t)&(\|V_{x}(t)\|_{1}^{2}+\|V_{xt}(t)\|^{2})+\int_{0}^{t}\left[(1+\tau)(\|V_{xx}(\tau)\|^{2}+\|V_{xt}(\tau)\|^{2})\right]{\rm d}\tau \\&
   \leq  C\left(\|V_{0}\|^{2}_{2}+\|z_{0}\|_{1}^{2}+ \delta \right).
   \end{split}
\end{equation}
Here we have used \eqref{3.4} and \eqref{3.32}. Moreover, multiplying \eqref{3.30} by $(1+t)^{2}$ and integrating it over $[0,t]$ gives
\begin{equation}\label{3.39}
 \begin{split}
   (1+t)^{2}&(\|V_{xx}(t)\|^{2}+\|V_{xt}(t)\|^{2})+\int_{0}^{t}(1+\tau)^{2}\|V_{xt}(\tau)\|^{2}{\rm d}\tau \\&
   \leq  C\left(\|V_{0}\|^{2}_{2}+\|z_{0}\|_{1}^{2}+ \delta \right).
   \end{split}
 \end{equation}
 Here we have used \eqref{3.4}, \eqref{3.32} and \eqref{3.38}.
 Combining two above equations, we can obtain \eqref{3.36}. The proof of Lemma \ref{L3.3} is completed.
\end{proof}
\begin{lemma}\label{L3.4}
 Under the assumptions of Proposition \ref{p1}, if $N(T) \leq \varepsilon^{2}$ and $\delta$ are small enough, it holds that
\begin{equation}\label{3.40}
\begin{split}
   (1+t)^{3}&(\|V_{xxx}(t)\|^{2}+\|V_{xxt}(t)\|^{2})+\int_{0}^{t}\left[(1+\tau)^{2}\|V_{xxx}(\tau)\|^{2}+(1+\tau)^{3}\|V_{xxt}(\tau)\|^{2}\right]{\rm d}\tau \\&
   \leq  C\left(\|V_{0}\|^{2}_{3}+\|z_{0}\|_{2}^{2}+ \delta \right),
   \end{split}
  \end{equation}
  for $0 \leq t \leq T$.
\end{lemma}
\begin{proof}
In a fashion similar to that above, multiplying $\eqref{1.23}_{1x}$ by $-V_{xxx}$ and integrating it with respect to $x$ over $\mathbb{R}$, we have after some integrations by parts that
\begin{equation}\label{3.41}
\begin{split}
&\frac{1}{2}\frac{{\rm d}}{{\rm d}t}\int_{\mathbb{R}} \left(V_{xx}^{2}+2V_{xx}V_{xxt}\right){\rm d} x-\int_{\mathbb{R}} p^\prime(\bar{v})V_{xxx}^{2} {\rm d} x \\
&=\int_{\mathbb{R}}V_{xxt}^{2}{\rm d}x+\int_{\mathbb{R}}[2p^\prime(\bar{v})_{x}V_{xx}+p^\prime(\bar{v})_{xx}V_{x}-F_{1x}]V_{xxx}{\rm d}x-\int_{\mathbb{R}}F_{2x}V_{xxx}{\rm d}x.
\end{split}
\end{equation}
The right-hand side of \eqref{3.41} can be estimated as follows. From \eqref{2.6}-\eqref{3.2}, one has that
  \begin{equation}\label{3.42}
  \begin{split}
   &\int_{\mathbb{R}}[2p^\prime(\bar{v})_{x}V_{xx}+p^\prime(\bar{v})_{xx}V_{x}-F_{1x}]V_{xxx}{\rm d}x \\
   \leq& -\frac{p^\prime(\bar{v})}{16}\|V_{xxx}(t)\|^{2}+C\delta (1+t)^{-\frac{7}{2}}+C\delta(1+t)^{-2}\|V_{x}(t)\|^{2}+C(1+t)^{-1}\|V_{xx}(t)\|^{2}.
   \end{split}
  \end{equation}
Noticing that
\begin{align}\label{3.43}
 -\int_{\mathbb{R}}F_{2x}V_{xxx}{\rm d}x=&-\int_{\mathbb{R}}\left(g^\prime fV_{xt}\right)_{x}V_{xxx}{\rm d}x-\int_{\mathbb{R}}\left(gf^\prime V_{xx}\right)_{x}V_{xxx}{\rm d}x\nonumber\\
    &-\int_{\mathbb{R}}[g^\prime f(-p(\bar{v})_{xx}+\hat{v}_{t})]_{x}V_{xxx}{\rm d}x-\int_{\mathbb{R}}[gf^\prime(\bar{v}_{x}+\hat{v}_{x})]_{x}V_{xxx}{\rm d}x\nonumber\\
    &:= I_{6}+I_{7}+I_{8}+I_{9},
\end{align}
then from \eqref{2.6}-\eqref{3.3}, we can conclude that
\begin{equation}\label{3.44}
    \begin{split}
   I_{6}=&-\int_{\mathbb{R}}\left(g^\prime f\right)_{x}V_{xt}V_{xxx}{\rm d}x-\int_{\mathbb{R}}g^\prime fV_{xxt}V_{xxx}{\rm d}x\\
   &\leq -\frac{p^\prime(\bar{v})}{16}\|V_{xxx}(t)\|^{2}+C\|V_{xxt}(t)\|^{2}+C(1+t)^{-2}\|V_{xt}(t)\|^{2},
    \end{split}
   \end{equation}
   \begin{equation}\label{3.45}
    \begin{split}
    I_{7}=& \int_{\mathbb{R}}\left(gf^\prime\right)_{x}V_{xx}V_{xxx}{\rm d}x+\int_{\mathbb{R}}gf^\prime V_{xxx}^{2}{\rm d}x\\
    \leq& -\frac{p^\prime(\bar{v})}{16}\|V_{xxx}(t)\|^{2}+C(1+t)^{-3}\|V_{xx}(t)\|^{2}+C{\rm e}^{-t}\|V_{xxx}(t)\|^{2},
    \end{split}
   \end{equation}
   and
   \begin{equation}\label{3.46}
    \begin{split}
    I_{8}+I_{9} &\leq C(1+t)^{-\frac{1}{2}}\int_{\mathbb{R}}(|\bar{v}_{xxx}|+|\bar{v}_{x}||\bar{v}_{xx}|+|\bar{v}_{x}|^{3}+|\hat{v}_{xt}|)|V_{xxx}|{\rm d}x\\
    &~~~~+C(1+t)^{-\frac{3}{2}}\int_{\mathbb{R}}(|\bar{v}_{x}|+|\hat{v}_{x}|)|V_{xxx}|{\rm d}x\\
    &~~~~+C(1+t)^{-1}\int_{\mathbb{R}}(|\bar{v}_{xx}|+|\bar{v}_{x}|^{2}+|\hat{v}_{t}|+|\hat{v}_{xx}|)|V_{xxx}|{\rm d}x\\
    &\leq -\frac{p^\prime(\bar{v})}{16}\|V_{xxx}(t)\|^{2}+C\delta (1+t)^{-\frac{7}{2}}.
    \end{split}
   \end{equation}
  Substituting \eqref{3.42} and \eqref{3.44}-\eqref{3.46} into \eqref{3.41}, we have
  \begin{equation}\label{3.47}
  \begin{split}
     &\frac{1}{2}\frac{{\rm d}}{{\rm d}t}\int_{\mathbb{R}}(V_{xx}^{2}+V_{xx}V_{xxt}){\rm d}x-\frac{3}{4}\int_{\mathbb{R}}p^\prime(\bar{v})V_{xxx}^{2}{\rm d}x\\
     \leq&  C\|V_{xxt}(t)\|^{2}+C{\rm e}^{-t}\|V_{xxx}(t)\|^{2}+C(1+t)^{-2}(\|V_{x}(t)\|^{2}+\|V_{xt}(t)\|^{2})\\&
     +C(1+t)^{-1}\|V_{xx}(t)\|^{2}+C\delta (1+t)^{-\frac{7}{2}}.
     \end{split}
  \end{equation}
Then the calculations of $\int_{\mathbb{R}}\partial_{x}^{2}\eqref{1.23}_{1}\times V_{xxt}{\rm d}x$ gives
\begin{equation}\label{3.48}
 \begin{split}
   &\frac{1}{2}\frac{{\rm d}}{{\rm d}t}\int_{\mathbb{R}}\left(V_{xxt}^{2}-p^\prime(\bar{v})V_{xxx}^{2}\right){\rm d}x+\int_{\mathbb{R}}V_{xxt}^{2}{\rm d}x\\
   =&-\frac{1}{2}\int_{\mathbb{R}}p^{\prime\prime}(\bar{v})\bar{v}_{t}V_{xxx}^{2}{\rm d}x+\int_{\mathbb{R}}\left[F_{1x}-p^\prime(\bar{v})_{xx}V_{x}-2p^\prime(\bar{v})_{x}V_{xx}\right]_{x}V_{xxt}{\rm d}x+\int_{\mathbb{R}}F_{2xx}V_{xxt}{\rm d}x.
    \end{split}
\end{equation}
By using \eqref{2.6}, we first have
\begin{equation}\label{3.49}
 -\frac{1}{2}\int_{\mathbb{R}}p^{\prime\prime}(\bar{v})\bar{v}_{t}V_{xxx}^{2}{\rm d}x\leq C\delta(1+t)^{-1}\|V_{xxx}(t)\|^{2}.	
\end{equation}
 As shown in \cite{Nishihara1996}, a directly calculation shows that
\begin{equation}\label{3.50}
 \begin{split}
   &\int_{\mathbb{R}}\left[F_{1x}-p^\prime(\bar{v})_{xx}V_{x}-2p^\prime(\bar{v})_{x}V_{xx}\right]_{x}V_{xxt}{\rm d}x\\
    \leq& \frac{1}{16} \|V_{xxt}(t)\|^{2}+\frac{1}{2}\frac{{\rm d}}{{\rm d}t}\int_{\mathbb{R}}\left[p^\prime(V_{x}+\bar{v}+\hat{v})-p^\prime(\bar{v})\right]V_{xxx}^{2}{\rm d}x+C\delta (1+t)^{-\frac{9}{2}}+C(1+t)^{-3}\|V_{x}(t)\|^{2}\\
    &+C(\varepsilon +\delta)(1+t)^{-1}\|V_{xxx}(t)\|^{2}+C(1+t)^{-2}\|V_{xx}(t)\|^{2}.
    \end{split}
\end{equation}
 Now we turn to estimate the last term in the right hand of \eqref{3.48}. Notice that
 \begin{align}\label{3.51}
\int_{\mathbb{R}}F_{2xx}V_{xxt}{\rm d}x=&\int_{\mathbb{R}}(g^\prime fV_{xt})_{xx}V_{xxt}{\rm d}x+\int_{\mathbb{R}}(gf^\prime V_{xx})_{xx}V_{xxt}{\rm d}x\nonumber\\
&+\int_{\mathbb{R}}(g^\prime f(-p(\bar{v})_{xx}+\hat{v}_{t}))_{xx}V_{xxt}{\rm d}x\nonumber\\
&+\int_{\mathbb{R}}(gf^\prime(\bar{v}_{x}+\hat{v}_{x}))_{xx}V_{xxt}{\rm d}x:=I_{10}+I_{11}+I_{12}+I_{13}.
\end{align}
 From \eqref{2.6}-\eqref{2.8} and {\it a priori} assumption \eqref{3.1}-\eqref{3.3}, we can deduce that
  \begin{equation}\label{3.52}
  	\begin{split}
  	I_{10}&=\int_{\mathbb{R}}\left(g^\prime f\right)_{xx}V_{xt}V_{xxt}{\rm d}x+2\int_{\mathbb{R}}\left(g^\prime f\right)_{x}V_{xxt}^{2}{\rm d}x+\int_{\mathbb{R}}g^\prime fV_{xxt}V_{xxxt}{\rm d}x\\
  	&\leq \frac{1}{32}\int_{\mathbb{R}}V_{xxt}^{2}{\rm d}x+C(1+t)^{-3}\int_{\mathbb{R}}V_{xt}^{2}{\rm d}x+C\varepsilon (1+t)^{-\frac{9}{4}}\int_{\mathbb{R}}|V_{xxt}||V_{xxx}|{\rm d}x-\frac{1}{2}\int_{\mathbb{R}}\left(g^\prime f\right)_{x}V_{xxt}^{2}{\rm d}x\\
  	&\leq \frac{1}{16} \|V_{xxt}(t)\|^{2}+C(1+t)^{-3}\|V_{xt}(t)\|^{2}+C\varepsilon	(1+t)^{-\frac{9}{2}}\|V_{xxx}(t)\|^{2},
  	\end{split}
  \end{equation}
  and
  \begin{equation}\label{3.53}
  \begin{split}
  	I_{11}&=\int_{\mathbb{R}}\left(gf^\prime\right)_{xx}V_{xx}V_{xxt}{\rm d}x+2\int_{\mathbb{R}}\left(gf^\prime\right)_{x}V_{xxx}V_{xxt}{\rm d}x-\int_{\mathbb{R}}gf^\prime V_{xxxx}V_{xxt}{\rm d}x\\
  	&\leq \frac{1}{32}\int_{\mathbb{R}}V_{xxt}^{2}{\rm d}x+C(1+t)^{-4}\int_{\mathbb{R}}V_{xx}^{2}{\rm d}x+C\varepsilon(1+t)^{-1}\int_{\mathbb{R}}|V_{xxt}||V_{xxx}|{\rm d}x\\
  	&~~~~+C(\varepsilon +\delta)(1+t)^{-3}\int_{\mathbb{R}}V_{xxx}^{2}{\rm d}x-\frac{1}{2}\frac{{\rm d}}{{\rm d}t}\int_{\mathbb{R}}gf^\prime V_{xxx}^{2}{\rm d}x+\frac{1}{2}\int_{\mathbb{R}}\left(gf^\prime\right)_{t}V_{xxx}^{2}{\rm d}x\\
  	&\leq \frac{1}{16} \|V_{xxt}(t)\|^{2}+C(1+t)^{-4}\|V_{xx}(t)\|^{2}+C(\varepsilon +\delta)(1+t)^{-2}\|V_{xxx}(t)\|^{2}\\
  	&~~~~+C{\rm e}^{-t}\|V_{xxx}(t)\|^{2}-\frac{1}{2}\frac{{\rm d}}{{\rm d}t}\int_{\mathbb{R}}gf^\prime V_{xxx}^{2}{\rm d}x.
  	\end{split}	
  \end{equation}
  Furthermore, we can similarly prove
  \begin{equation}\label{3.54}
  	\begin{split}
  	I_{12}+I_{13} \leq \frac{1}{16} \|V_{xxt}(t)\|^{2}+C\delta(1+t)^{-3}\|V_{xxx}(t)\|^{2}+C\delta(1+t)^{-\frac{9}{2}}.	
  	\end{split}
  \end{equation}
Substituting \eqref{3.49}-\eqref{3.50} and \eqref{3.52}-\eqref{3.54} into \eqref{3.48}, we have
\begin{equation}\label{3.55}
\begin{split}
    &\frac{1}{2}\frac{{\rm d}}{{\rm d}t}\int_{\mathbb{R}}\left[V_{xxt}^{2}+\left(gf^\prime-p^\prime(\bar{v})\right)V_{xxx}^{2}\right]{\rm d}x+\frac{3}{4}\int_{\mathbb{R}}V_{xxt}^{2}{\rm d}x \\
    \leq& \frac{1}{2}\frac{{\rm d}}{{\rm d}t}\int_{\mathbb{R}}\left[p^\prime(V_{x}+\bar{v}+\hat{v})-p^\prime(\bar{v})\right]V_{xxx}^{2}{\rm d}x+C(\varepsilon +\delta)(1+t)^{-1}\|V_{xxx}(t)\|^{2}+C\delta (1+t)^{-\frac{9}{2}}\\&
    +C(1+t)^{-2}\|V_{xx}(t)\|^{2}+C(1+t)^{-3}(\|V_{x}(t)\|^{2}+\|V_{xt}(t)\|^{2})+C{\rm e}^{-t}\|V_{xxx}(t)\|^{2}.
    \end{split}
\end{equation}
  Addition of $\lambda \cdot \eqref{3.47}$ to \eqref{3.55} $(0<\lambda \ll 1)$, one has that
  \begin{equation}\label{3.56}
   \begin{split}
   &\frac{1}{2}\frac{{\rm d}}{{\rm d}t}\int_{\mathbb{R}} \left[V_{xxt}^{2}+\lambda V_{xx}^{2}+2\lambda V_{xxt}V_{xx}+\left(gf^\prime-p^\prime(\bar{v})\right)V_{xxx}^{2}\right]{\rm d}x+\frac{1}{2}\int_{\mathbb{R}} \left(V_{xxt}^{2}-\lambda p^\prime(\bar{v})V_{xxx}^{2}\right){\rm d} x \\
   \leq& \frac{1}{2}\frac{{\rm d}}{{\rm d}t}\int_{\mathbb{R}}\left[p^\prime(V_{x}+\bar{v}+\hat{v})-p^\prime(\bar{v})\right]V_{xxx}^{2}{\rm d}x+ C\delta (1+t)^{-\frac{7}{2}}+C(1+t)^{-1}\|V_{xx}(t)\|^{2}\\&
   +C(1+t)^{-2}(\|V_{x}(t)\|^{2}+\|V_{xt}(t)\|^{2})+C{\rm e}^{-t}\|V_{xxx}(t)\|^{2} .
    \end{split}
    \end{equation}
  Integrating \eqref{3.56} over $[0,t]$ and using Lemma \ref{L3.2}, one gets
\begin{equation}\label{3.57}
  \begin{split}
   \|V_{xx}(t)\|_{1}^{2}+&\|V_{xxt}(t)\|^{2}+\int_{0}^{t}(\|V_{xxx}(\tau)\|^{2}+\|V_{xxt}(\tau)\|^{2}){\rm d}\tau \\&
   \leq  C\left(\|V_{0}\|^{2}_{3}+\|z_{0}\|_{2}^{2}+ \delta \right)+C\int_{0}^{t}{\rm e}^{-\tau}\|V_{xxx}(\tau)\|^{2}{\rm d}\tau .
    \end{split}
  \end{equation}
  It follows from Gronwall's inequality that
  \begin{equation}\label{3.58}
   \begin{split}
   \|V_{xx}(t)\|_{1}^{2}+&\|V_{xxt}(t)\|^{2}+\int_{0}^{t}(\|V_{xxx}(\tau)\|^{2}+\|V_{xxt}(\tau)\|^{2}){\rm d}\tau \\&
   \leq  C\left(\|V_{0}\|^{2}_{3}+\|z_{0}\|_{2}^{2}+ \delta \right) .
    \end{split}
  \end{equation}
Integrating $(1+t)\cdot \eqref{3.56}$ and $(1+t)^{2}\cdot \eqref{3.56}$ over $[0,t]$, we have
\begin{equation}\label{3.59}
   \begin{split}
   (1+t)^{2}(\|V_{xx}(t)\|_{1}^{2}+&\|V_{xxt}(t)\|^{2})+\int_{0}^{t}(1+\tau)^{2}(\|V_{xxx}(\tau)\|^{2}+\|V_{xxt}(\tau)\|^{2}){\rm d}\tau \\&
   \leq  C\left(\|V_{0}\|^{2}_{3}+\|z_{0}\|_{2}^{2}+ \delta \right) .
    \end{split}
  \end{equation}
  Here we have used Lemma \ref{L3.2}-Lemma \ref{L3.3}. Then the integration of $(1+t)^{3} \cdot \eqref{3.55}$ over $[0,t]$ yields
  \begin{equation}\label{3.60}
   (1+t)^{3}(\|V_{xxx}(t)\|^{2}+\|V_{xxt}(t)\|^{2})+\int_{0}^{t}(1+\tau)^{3}\|V_{xxt}(\tau)\|^{2}{\rm d}\tau
   \leq  C\left(\|V_{0}\|^{2}_{3}+\|z_{0}\|_{2}^{2}+ \delta \right) .
  \end{equation}
  Combining \eqref{3.59} and \eqref{3.60}, one can immediately obtain \eqref{3.40}. Thus the proof of Lemma \ref{L3.4} is completed.
\end{proof}
\begin{lemma}\label{L3.5}
 Under the assumptions of Proposition \ref{p1}, if $N(T) \leq \varepsilon^{2}$ and $\delta$ are small enough, it holds that
 \begin{equation}\label{3.61}
 \begin{split}
   (1+t)^{2}\|V_{t}(t)\|^{2}&+(1+t)^{3}(\|V_{tt}(t)\|^{2}+\|V_{xt}(t)\|^{2})\\
   &+\int_{0}^{t}\left[(1+\tau)^{2}\|V_{xt}(\tau)\|^{2}+(1+\tau)^{3}\|V_{tt}(\tau)\|^{2}\right]{\rm d}\tau \\
   \leq&  C\left(\|V_{0}\|^{2}_{2}+\|z_{0}\|_{1}^{2}+ \delta \right),
   \end{split}
  \end{equation}
  for $0 \leq t \leq T$.
\end{lemma}
\begin{proof}
Firstly, having $\int_{\mathbb{R}}V_{t}\times \eqref{1.23}_{1t}{\rm d} x$, we obtain
\begin{equation}\label{3.62}
\frac{1}{2}\frac{{\rm d}}{{\rm d}t}\int_{\mathbb{R}} \left(V_{t}^{2}+2V_{t}V_{tt}\right){\rm d} x-\int_{\mathbb{R}} p^\prime(\bar{v})V_{xt}^{2}{\rm d} x =\int_{\mathbb{R}}V_{tt}^{2}{\rm d}x+\int_{\mathbb{R}}(F_{1t}V_{t}+p^\prime(\bar{v})_{t}V_{x}V_{xt}){\rm d}x+\int_{\mathbb{R}}F_{2t}V_{t}{\rm d}x.
\end{equation}
By direct calculation, we first have
  \begin{equation}\label{3.63}
  \begin{split}
   &\int_{\mathbb{R}}(F_{1t}V_{t}+p^\prime(\bar{v})_{t}V_{x}V_{xt}){\rm d}x\\&
   \leq -\frac{p^\prime(\bar{v})}{16}\|V_{xt}(t)\|^{2}+C\delta (1+t)^{-\frac{7}{2}}+C\delta(1+t)^{-2}\|V_{x}(t)\|^{2},
   \end{split}
  \end{equation}
  then, notice that
  \begin{equation}\label{3.64}
  \begin{split}
  \int_{\mathbb{R}}F_{2t}V_{t}{\rm d}x&=\int_{\mathbb{R}}(g^\prime fV_{xt})_{t}V_{t}{\rm d}x+\int_{\mathbb{R}}(gf^\prime V_{xx})_{t}V_{t}{\rm d}x+\int_{\mathbb{R}}[g^\prime f(-p(\bar{v})_{xx}+\hat{v}_{t})]_{t}V_{t}{\rm d}x\\
  &~~~~+\int_{\mathbb{R}}[gf^\prime(\bar{v}_{x}+\hat{v}_{x})]_{t}V_{t}{\rm d}x\\
  &:= I_{14}+I_{15}+I_{16}+I_{17}.
    \end{split}
  \end{equation}
From \eqref{2.6}-\eqref{3.3} and Young's inequality, we have
  \begin{equation}\label{3.65}
    \begin{split}
    I_{14}&=\int_{\mathbb{R}}\left(g^\prime f\right)_{t}V_{xt}V_{t}{\rm d}x+\int_{\mathbb{R}}g^\prime fV_{xtt}V_{t}{\rm d}x\\
    &\leq -\frac{p^\prime(\bar{v})}{32}\|V_{xt}(t)\|^{2}+C(1+t)^{-3}\|V_{t}(t)\|^{2}-\int_{\mathbb{R}}g^\prime fV_{tt}V_{xt}{\rm d}x-\int_{\mathbb{R}}\left(g^\prime f\right)_{x}V_{tt}V_{t}{\rm d}x\\
    &\leq -\frac{p^\prime(\bar{v})}{16}\|V_{xt}(t)\|^{2}+C(1+t)^{-2}\|V_{t}(t)\|^{2}+C\|V_{tt}(t)\|^{2},
    \end{split}
  \end{equation}
  \begin{equation}\label{3.66}
    \begin{split}
    I_{15}&=\int_{\mathbb{R}}\left(gf^\prime\right)_{t}V_{xx}V_{t}{\rm d}x+\int_{\mathbb{R}}gf^\prime V_{xxt}V_{t}{\rm d}x\\
    &\leq C(1+t)^{-1}\|V_{t}(t)\|^{2}+C(1+t)^{-3}\|V_{xx}(t)\|^{2}-\int_{\mathbb{R}}gf^\prime V_{xt}^{2}{\rm d}x-\int_{\mathbb{R}}\left(gf^\prime\right)_{x}V_{xt}V_{t}{\rm d}x\\
    &\leq C(1+t)^{-1}\|V_{t}(t)\|^{2}+C(1+t)^{-3}\|V_{xx}(t)\|^{2}-\frac{p^\prime(\bar{v})}{16}\|V_{xt}(t)\|^{2}+C{\rm e}^{-t}\|V_{xt}(t)\|^{2},
    \end{split}
  \end{equation}
  and we can similarly prove
  \begin{equation}\label{3.67}
    I_{16}+I_{17} \leq C(1+t)^{-1}\|V_{t}(t)\|^{2}+C\|V_{tt}(t)\|^{2}+C\delta (1+t)^{-\frac{7}{2}}.
  \end{equation}
  Substituting \eqref{3.63} and \eqref{3.65}-\eqref{3.67} into \eqref{3.62}, we have
  \begin{equation}\label{3.68}
  \begin{split}
     &\frac{1}{2}\frac{{\rm d}}{{\rm d}t}\int_{\mathbb{R}}(V_{t}^{2}+2V_{t}V_{tt}){\rm d}x -\frac{1}{2}\int_{\mathbb{R}}p^\prime(\bar{v})V_{xt}^{2}{\rm d}x\\
      \leq&  C\|V_{tt}(t)\|^{2}+C{\rm e}^{-t}\|V_{xt}(t)\|^{2}+C(1+t)^{-2}\|V_{x}(t)\|^{2}+C(1+t)^{-3}\|V_{xx}(t)\|^{2}+C\delta (1+t)^{-\frac{7}{2}}\\&+C(1+t)^{-1}\|V_{t}(t)\|^{2}.
     \end{split}
  \end{equation}
  Next, by calculating $\int_{\mathbb{R}}V_{tt}\times \eqref{1.23}_{1t}{\rm d} x$, we have
\begin{equation}\label{3.69}
\begin{split}
  &\frac{1}{2}\frac{{\rm d}}{{\rm d}t}\int_{\mathbb{R}}(V_{tt}^{2}-p^\prime(\bar{v})V_{xt}^{2}){\rm d}x +\int_{\mathbb{R}}V_{tt}^{2}{\rm d}x \\
  =&-\frac{1}{2}\int_{\mathbb{R}}p^\prime(\bar{v})_{t}V_{xt}^{2}{\rm d}x+\int_{\mathbb{R}}\left[F_{1t}-\left(p^\prime(\bar{v})_{t}V_{x}\right)_{x}\right]V_{tt}{\rm d}x+\int_{\mathbb{R}}F_{2t}V_{tt}{\rm d}x .
  \end{split}
  \end{equation}
  The right hand side of \eqref{3.69} can be estimated as follows. From Lemma \ref{L2.1}, we have
  \begin{equation}\label{3.70}
  -\frac{1}{2}\int_{\mathbb{R}}p^\prime(\bar{v})_{t}V_{xt}^{2}{\rm d}x \leq C\delta(1+t)^{-1}\|V_{xt}(t)\|^{2}.	
  \end{equation}
  By direct calculation, we have
  \begin{equation}\label{3.71}
  \begin{split}
  &\int_{\mathbb{R}}\left[F_{1t}-\left(p^\prime(\bar{v})_{t}V_{x}\right)_{x}\right]V_{tt}{\rm d}x \\
   \leq& \frac{1}{16}\|V_{tt}(t)\|^{2}+\frac{1}{2}\frac{{\rm d}}{{\rm d}t}\int_{\mathbb{R}}\left[p^\prime(V_{x}+\bar{v}+\hat{v})-p^\prime(\bar{v})\right]V_{xt}^{2}{\rm d}x+C\delta (1+t)^{-\frac{9}{2}}+C\delta(1+t)^{-3}\|V_{x}(t)\|^{2}\\&
   +C(\delta+\varepsilon)(1+t)^{-1}\|V_{xt}(t)\|^{2}+C(1+t)^{-2}\|V_{xx}(t)\|^{2}.
   \end{split}
  \end{equation}
  Finally, notice that
  \begin{equation}\label{3.72}
  \begin{split}
  \int_{\mathbb{R}}F_{2t}V_{tt}{\rm d}x&=\int_{\mathbb{R}}(g^\prime fV_{xt})_{t}V_{tt}{\rm d}x+\int_{\mathbb{R}}(gf^\prime V_{xx})_{t}V_{tt}{\rm d}x+\int_{\mathbb{R}}[g^\prime f(-p(\bar{v})_{xx}+\hat{v}_{t})]_{t}V_{tt}{\rm d}x\\
  &~~~~+\int_{\mathbb{R}}[gf^\prime(\bar{v}_{x}+\hat{v}_{x})]_{t}V_{tt}{\rm d}x\\
  &:= I_{18}+I_{19}+I_{20}+I_{21}.
  \end{split}
   \end{equation}
   From \eqref{2.6}-\eqref{3.3} and Young's inequality, we get
  \begin{equation}\label{3.73}
   	\begin{split}
   I_{18}&=\int_{\mathbb{R}}\left(g^\prime f\right)_{t}V_{xt}V_{tt}{\rm d}x+\int_{\mathbb{R}}g^\prime fV_{xtt}V_{tt}{\rm d}x\\
   &\leq \frac{1}{16} \|V_{tt}(t)\|^{2}+C(\varepsilon +\delta)(1+t)^{-3}\|V_{xt}(t)\|^{2}+C{\rm e}^{-t}\|V_{xt}(t)\|^{2},	
   	\end{split}
   \end{equation}
   \begin{equation}\label{3.74}
    \begin{split}
    I_{19}&=\int_{\mathbb{R}}\left(gf^\prime\right)_{t}V_{xx}V_{tt}{\rm d}x+\int_{\mathbb{R}}gf^\prime V_{xxt}V_{tt}{\rm d}x\\
    &\leq \frac{1}{32} \|V_{tt}(t)\|^{2}+C(1+t)^{-4}\|V_{xx}(t)\|^{2}-\frac{1}{2}\frac{{\rm d}}{{\rm d}t}\int_{\mathbb{R}}gf^\prime V_{xt}^{2}{\rm d}x+\frac{1}{2}\int_{\mathbb{R}}\left(gf^\prime\right)_{t}V_{xt}^{2}{\rm d}x\\
    &~~~~-\int_{\mathbb{R}}\left(gf^\prime\right)_{x}V_{xt}V_{tt}{\rm d}x\\
    &\leq \frac{1}{16} \|V_{tt}(t)\|^{2}+C(1+t)^{-4}\|V_{xx}(t)\|^{2}+C(\varepsilon +\delta)(1+t)^{-2}\|V_{xt}(t)\|^{2}\\
    &~~~~-\frac{1}{2}\frac{{\rm d}}{{\rm d}t}\int_{\mathbb{R}}gf^\prime V_{xt}^{2}{\rm d}x+C{\rm e}^{-t}\|V_{xt}(t)\|^{2},		
    	\end{split}	
    \end{equation}
  and
  \begin{equation}\label{3.75}
  \begin{split}
  I_{20}+I_{21} &\leq C(1+t)^{-\frac{1}{2}}\int_{\mathbb{R}}(|\bar{v}_{xxt}|+|\bar{v}_{xx}||\bar{v}_{t}|+|\bar{v}_{x}||\bar{v}_{xt}|+|\bar{v}_{x}|^{2}|\bar{v}_{t}|+|\hat{v}_{tt}|)|V_{tt}|{\rm d}x\\
  &~~~~+C(1+t)^{-\frac{3}{2}}\int_{\mathbb{R}}(|\bar{v}_{xx}|+|\bar{v}_{x}|^{2}+|\hat{v}_{t}|)|V_{tt}|{\rm d}x+C(1+t)^{-1}\int_{\mathbb{R}}(|\bar{v}_{xt}|+|\hat{v}_{xt}|)|V_{tt}|{\rm d}x\\
  &~~~~+C(1+t)^{-2}\int_{\mathbb{R}}(|\bar{v}_{x}|+|\hat{v}_{x}|)|V_{tt}|{\rm d}x\\
  &\leq \frac{1}{16} \|V_{tt}(t)\|^{2}+C\delta(1+t)^{-\frac{9}{2}}.
  \end{split}	
  \end{equation}
Substituting \eqref{3.70}-\eqref{3.71} and \eqref{3.73}-\eqref{3.75} into \eqref{3.69}, we have
\begin{equation}\label{3.76}
\begin{split}
     &\frac{1}{2}\frac{{\rm d}}{{\rm d}t}\int_{\mathbb{R}}\left[V_{tt}^{2}+\left(gf^\prime-p^\prime(\bar{v})\right)V_{xt}^{2}\right]{\rm d}x+\frac{3}{4}\int_{\mathbb{R}}V_{tt}^{2}{\rm d}x \\
    \leq& \frac{1}{2}\frac{{\rm d}}{{\rm d}t}\int_{\mathbb{R}}\left[p^\prime(V_{x}+\bar{v}+\hat{v})-p^\prime(\bar{v})\right]V_{xt}^{2}{\rm d}x+C(\varepsilon +\delta)(1+t)^{-1}\|V_{xt}(t)\|^{2}+C\delta (1+t)^{-\frac{9}{2}}\\&
    +C(1+t)^{-2}\|V_{xx}(t)\|^{2}+C(1+t)^{-3}\|V_{x}(t)\|^{2}+C{\rm e}^{-t}\|V_{xt}(t)\|^{2}.
    \end{split}
\end{equation}
   Addition of $\lambda \cdot \eqref{3.68}$ to \eqref{3.76} $(0<\lambda\ll 1)$, one has that
  \begin{equation}\label{3.77}
   \begin{split}
   &\frac{1}{2}\frac{{\rm d}}{{\rm d}t}\int_{\mathbb{R}} \left[V_{tt}^{2}+\lambda V_{t}^{2}+2\lambda V_{t}V_{tt}+\left(gf^\prime-p^\prime(\bar{v})\right)V_{xt}^{2}\right]{\rm d}x+\frac{1}{2}\int_{\mathbb{R}} \left(V_{tt}^{2}-\lambda p^\prime(\bar{v})V_{xt}^{2}\right){\rm d} x \\
   \leq& \frac{1}{2}\frac{{\rm d}}{{\rm d}t}\int_{\mathbb{R}}\left[p^\prime(V_{x}+\bar{v}+\hat{v})-p^\prime(\bar{v})\right]V_{xt}^{2}{\rm d}x+ C\delta (1+t)^{-\frac{7}{2}}+C(1+t)^{-1}\|V_{t}(t)\|^{2}\\&
   ~~~~+C(1+t)^{-2}\|V_{x}(t)\|_{1}^{2}+C{\rm e}^{-t}\|V_{xt}(t)\|^{2}.
    \end{split}
    \end{equation}
  Integrating \eqref{3.77}, $(1+t)\cdot \eqref{3.77}$ and $(1+t)^{2}\cdot \eqref{3.77}$ over $[0,t]$ and using Lemma \ref{L3.2}-Lemma \ref{L3.4}, one gets
\begin{equation}\label{3.78}
   (1+t)^{2}(\|V_{t}(t)\|_{1}^{2}+\|V_{tt}(t)\|^{2})+\int_{0}^{t}(1+\tau)^{2}(\|V_{tt}(\tau)\|^{2}+\|V_{xt}(\tau)\|^{2}){\rm d}\tau
   \leq  C\left(\|V_{0}\|^{2}_{2}+\|z_{0}\|_{1}^{2}+ \delta \right).
  \end{equation}
  Then the integration of $(1+t)^{3} \cdot \eqref{3.76}$ over $[0,t]$ yields
  \begin{equation}\label{3.79}
   (1+t)^{3}(\|V_{xt}(t)\|^{2}+\|V_{tt}(t)\|^{2})+\int_{0}^{t}(1+\tau)^{3}\|V_{tt}(\tau)\|^{2}{\rm d}\tau
   \leq  C\left(\|V_{0}\|^{2}_{2}+\|z_{0}\|_{1}^{2}+ \delta \right) .
  \end{equation}
  Combining \eqref{3.78} and \eqref{3.79}, one can immediately obtain \eqref{3.61}. The proof of Lemma \ref{L3.5} is completed.
\end{proof}
\begin{lemma}\label{L3.6}
 Under the assumptions of Proposition \ref{p1}, if $N(T) \leq \varepsilon^{2}$ and $\delta$ are small enough, it holds that
 \begin{equation}\label{3.80}
 \begin{split}
   (1+t)^{4}&(\|V_{xtt}(t)\|^{2}+\|V_{xxt}(t)\|^{2})+\int_{0}^{t}\left[(1+\tau)^{3}\|V_{xxt}(\tau)\|^{2}+(1+\tau)^{4}\|V_{xtt}(\tau)\|^{2}\right]{\rm d}\tau \\&
   \leq  C\left(\|V_{0}\|^{2}_{3}+\|z_{0}\|_{2}^{2}+ \delta \right),
   \end{split}
  \end{equation}
  for $0 \leq t \leq T$.
\end{lemma}
\begin{proof}
 Having $\int_{\mathbb{R}}\partial_{xt}\eqref{1.23}_{1}\times V_{xtt}{\rm d}x$, we obtain
\begin{equation}\label{3.81}
\begin{split}
 &\frac{1}{2}\frac{{\rm d}}{{\rm d}t}\int_{\mathbb{R}}(V_{xtt}^{2}-p^\prime(\bar{v})V_{xxt}^{2}){\rm d}x +\int_{\mathbb{R}}V_{xtt}^{2}{\rm d}x \\
  =&-\frac{1}{2}\int_{\mathbb{R}}p^\prime(\bar{v})_{t}V_{xxt}^{2}{\rm d}x+\int_{\mathbb{R}}\left[F_{1xt}-\left(p^\prime(\bar{v})_{t}V_{xx}+(p^\prime(\bar{v})_{x}V_{x})_{t}\right)_{x}\right]V_{xtt}{\rm d}x+\int_{\mathbb{R}}F_{2xt}V_{xtt}{\rm d}x.
  \end{split}
  \end{equation}
  We estimate the right hand side of \eqref{3.81} as follows. From Lemma \ref{L2.1}, we have
  \begin{equation}\label{3.82}
   -\frac{1}{2}\int_{\mathbb{R}}p^\prime(\bar{v})_{t}V_{xxt}^{2}{\rm d}x \leq C\delta(1+t)^{-1}\|V_{xxt}(t)\|^{2}.	
   \end{equation}
  A directly calculation shows that
  \begin{equation}\label{3.83}
  \begin{split}
   &\int_{\mathbb{R}}\left[F_{1xt}-\left(p^\prime(\bar{v})_{t}V_{xx}+(p^\prime(\bar{v})_{x}V_{x})_{t}\right)_{x}\right]V_{xtt}{\rm d}x \\
   \leq& \frac{1}{16}\|V_{xtt}(t)\|^{2}+\frac{1}{2}\frac{{\rm d}}{{\rm d}t}\int_{\mathbb{R}}\left[p^\prime(V_{x}+\bar{v}+\hat{v})-p^\prime(\bar{v})\right]V_{xxt}^{2}{\rm d}x+C\delta (1+t)^{-\frac{11}{2}}+C(1+t)^{-2}\|V_{xt}(t)\|^{2}\\&
   +C(1+t)^{-3}\|V_{xx}(t)\|^{2}+C(1+t)^{-2}\|V_{xxx}(t)\|^{2}+C(\delta+\varepsilon)(1+t)^{-1}\|V_{xxt}(t)\|^{2}\\&
   +C\delta(1+t)^{-4}\|V_{x}(t)\|^{2}.
   \end{split}
  \end{equation}
  Now we deal with the last term of the righthand side of \eqref{3.81}. Notice
  \begin{equation}\label{3.84}
  	\begin{split}
  	\int_{\mathbb{R}}F_{2xt}V_{xtt}{\rm d}x&=\int_{\mathbb{R}}\left(g^\prime fV_{xt}\right)_{xt}V_{xtt}{\rm d}x+\int_{\mathbb{R}}\left(gf^\prime V_{xx}\right)_{xt}V_{xtt}{\rm d}x+\int_{\mathbb{R}}(g^\prime f(-p(\bar{v})_{xt}+\hat{v}_{t}))_{xx}V_{xtt}{\rm d}x\\
&+\int_{\mathbb{R}}(gf^\prime(\bar{v}_{x}+\hat{v}_{x}))_{xt}V_{xtt}{\rm d}x:=I_{22}+I_{23}+I_{24}+I_{25}.	
  	\end{split}
  \end{equation}
  By using \eqref{2.6}-\eqref{3.3} and Young's inequality, we can get
  \begin{equation}\label{3.85}
  	\begin{split}
  	I_{22}&=\int_{\mathbb{R}}\left(g^\prime f\right)_{xt}V_{xt}V_{xtt}{\rm d}x+\int_{\mathbb{R}}\left(g^\prime f\right)_{x}V_{xtt}^{2}{\rm d}x+\int_{\mathbb{R}}\left(g^\prime f\right)_{t}V_{xxt}V_{xtt}{\rm d}x\\
  	&~~~~+\int_{\mathbb{R}}g^\prime fV_{xxtt}V_{xtt}{\rm d}x\\
  	&\leq \frac{1}{32} \|V_{xtt}(t)\|^{2}+C(1+t)^{-4}\|V_{xt}(t)\|^{2}+C(\varepsilon +\delta)(1+t)^{-3}\|V_{xxt}(t)\|^{2}+C{\rm e}^{-t}\|V_{xxt}(t)\|^{2}\\
  	&~~~~-\frac{1}{2}\int_{\mathbb{R}}\left(g^\prime f\right)_{x}V_{xtt}^{2}{\rm d}x\\
  	&\leq \frac{1}{16} \|V_{xtt}(t)\|^{2}+C(1+t)^{-4}\|V_{xt}(t)\|^{2}+C(\varepsilon +\delta)(1+t)^{-3}\|V_{xxt}(t)\|^{2}+C{\rm e}^{-t}\|V_{xxt}(t)\|^{2},	
  	\end{split}
  \end{equation}
  and
  \begin{equation}\label{3.86}
  \begin{split}
  	I_{23}&=\int_{\mathbb{R}}\left(gf^\prime\right)_{xt}V_{xx}V_{xtt}{\rm d}x+\int_{\mathbb{R}}\left(gf^\prime\right)_{x}V_{xxt}V_{xtt}{\rm d}x-\int_{\mathbb{R}}\left(gf^\prime\right)_{t}V_{xxx}V_{xtt}{\rm d}x\\
  	&~~~~+\int_{\mathbb{R}}gf^\prime V_{xxxt}V_{xtt}{\rm d}x\\
  	&\leq \frac{1}{32} \|V_{xtt}(t)\|^{2}+C(1+t)^{-5}\|V_{xx}(t)\|^{2}+C(\varepsilon +\delta)(1+t)^{-3}\|V_{xxt}(t)\|^{2}\\
  	&~~~~+C(1+t)^{-4}\|V_{xxx}(t)\|^{2}-\frac{1}{2}\frac{{\rm d}}{{\rm d}t}\int_{\mathbb{R}}gf^\prime V_{xxt}^{2}{\rm d}x+\frac{1}{2}\int_{\mathbb{R}}\left(gf^\prime\right)_{t}V_{xxt}^{2}{\rm d}x\\
  	&\leq \frac{1}{16} \|V_{xtt}(t)\|^{2}+C(1+t)^{-5}\|V_{xx}(t)\|^{2}+C(\varepsilon +\delta)(1+t)^{-2}\|V_{xxt}(t)\|^{2}\\
  	&~~~~+C(1+t)^{-4}\|V_{xxx}(t)\|^{2}-\frac{1}{2}\frac{{\rm d}}{{\rm d}t}\int_{\mathbb{R}}gf^\prime V_{xxt}^{2}{\rm d}x+C{\rm e}^{-t}\|V_{xxt}(t)\|^{2}.	
  	\end{split}	
  \end{equation}
  Furthermore, we can similarly prove
  \begin{equation}\label{3.87}
  	I_{24}+I_{25}\leq \frac{1}{16} \|V_{xtt}(t)\|^{2}+C\delta (1+t)^{-\frac{11}{2}}+C(\varepsilon +\delta)(1+t)^{-3}\|V_{xxt}(t)\|^{2}.
  \end{equation}
Substituting \eqref{3.82}-\eqref{3.83} and \eqref{3.85}-\eqref{3.87} into \eqref{3.81}, we have
\begin{equation}\label{3.88}
\begin{split}
    &\frac{1}{2}\frac{{\rm d}}{{\rm d}t}\int_{\mathbb{R}}\left[V_{xtt}^{2}+\left(gf^\prime-p^\prime(\bar{v})\right)V_{xxt}^{2}\right]{\rm d}x+\frac{3}{4}\int_{\mathbb{R}}V_{xtt}^{2}{\rm d}x \\
    \leq& \frac{1}{2}\frac{{\rm d}}{{\rm d}t}\int_{\mathbb{R}}\left[p^\prime(V_{x}+\bar{v}+\hat{v})-p^\prime(\bar{v})\right]V_{xxt}^{2}{\rm d}x+C(1+t)^{-1}\|V_{xxt}(t)\|^{2}+C\delta (1+t)^{-\frac{11}{2}}\\&
    ~~~~+C(1+t)^{-3}\|V_{xx}(t)\|^{2}+C(1+t)^{-4}\|V_{x}(t)\|^{2}+C(1+t)^{-2}\|V_{xxx}(t)\|^{2}\\&
    ~~~~+C(1+t)^{-2}\|V_{xt}(t)\|^{2}.
    \end{split}
\end{equation}
Integrating $(1+t)^{k} \cdot \eqref{3.88}$ over $[0,t]$, $k=0,1,2,3,4,$ and using Lemma \ref{L3.2}-Lemma \ref{L3.5}, one can immediately obtain \eqref{3.80}. The proof of Lemma \ref{L3.6} is completed.
\end{proof}
\begin{lemma}\label{L3.7}
 Under the assumptions of Proposition \ref{p1}, if $N(T) \leq \varepsilon^{2}$ and $\delta$ are small enough, it holds that
 \begin{equation}\label{3.89}
 \begin{split}
   (1+t)^{4}\|V_{tt}(t)\|^{2}&+(1+t)^{5}(\|V_{ttt}(t)\|^{2}+\|V_{xtt}(t)\|^{2})+\int_{0}^{t}(1+\tau)^{5}\|V_{ttt}(\tau)\|^{2}{\rm d}\tau\\
   \leq&  C\left(\|V_{0}\|^{2}_{3}+\|z_{0}\|_{2}^{2}+ \delta \right),
   \end{split}
  \end{equation}
  for $0 \leq t \leq T$.
\end{lemma}
\begin{proof}
By calculating $\int_{\mathbb{R}}V_{tt}\times \eqref{1.23}_{1tt}{\rm d}x$, we first have
\begin{equation}\label{3.90}
  \begin{split}
  \frac{1}{2}\frac{{\rm d}}{{\rm d}t}&\int_{\mathbb{R}} \left(V_{tt}^{2}+2V_{tt}V_{ttt}\right){\rm d} x-\int_{\mathbb{R}} p^\prime(\bar{v})V_{xtt}^{2}{\rm d} x\\
  &=\int_{\mathbb{R}}V_{ttt}^{2}{\rm d}x+\int_{\mathbb{R}}[F_{1tt}V_{tt}+(p^\prime(\bar{v})_{tt}V_{x}+2p^\prime(\bar{v})_{t}V_{xt})V_{xtt}]{\rm d}x+\int_{\mathbb{R}}F_{2tt}V_{tt}{\rm d}x.
  \end{split}
  \end{equation}
  It is not hard to get that
\begin{equation}\label{3.91}
 \begin{split}
 &\int_{\mathbb{R}}[F_{1tt}V_{tt}+(p^\prime(\bar{v})_{tt}V_{x}+2p^\prime(\bar{v})_{t}V_{xt})V_{xtt}]{\rm d}x\\
 \leq& -\frac{p^\prime(\bar{v})}{16}\|V_{xtt}(t)\|^{2}+C(1+t)^{-4}\|V_{x}(t)\|^{2}+C(1+t)^{-2}\|V_{xt}(t)\|^{2}+C\delta(1+t)^{-\frac{11}{2}},
 \end{split}
\end{equation}
and
\begin{align}\label{3.92}
 \int_{\mathbb{R}}F_{2tt}V_{tt}{\rm d}x=&\int_{\mathbb{R}}\left(g^\prime fV_{xt}\right)_{tt}V_{tt}{\rm d}x+\int_{\mathbb{R}}\left(gf^\prime V_{xx}\right)_{tt}V_{tt}{\rm d}x+\int_{\mathbb{R}}(g^\prime f(-p(\bar{v})_{xx}+\hat{v}_{t}))_{tt}V_{tt}{\rm d}x\nonumber\\
&+\int_{\mathbb{R}}(gf^\prime(\bar{v}_{x}+\hat{v}_{x}))_{tt}V_{tt}{\rm d}x\nonumber\\
\leq& -\frac{p^\prime(\bar{v})}{16}\|V_{xtt}(t)\|^{2}+C(1+t)^{-1}\|V_{tt}(t)\|^{2}+C(1+t)^{-4}\|V_{xt}(t)\|^{2}+C\|V_{ttt}(t)\|^{2}\nonumber\\
&+C(1+t)^{-5}\|V_{xx}(t)\|^{2}+C(1+t)^{-3}\|V_{xxt}(t)\|^{2}+C\delta(1+t)^{-\frac{11}{2}}+C{\rm e}^{-t}\|V_{xtt}(t)\|^{2}.
\end{align}
Substituting \eqref{3.91}-\eqref{3.92} into \eqref{3.90}, we have
\begin{equation}\label{3.93}
  \begin{split}
  &\frac{1}{2}\frac{{\rm d}}{{\rm d}t}\int_{\mathbb{R}} \left(V_{tt}^{2}+2V_{tt}V_{ttt}\right){\rm d} x-\frac{3}{4}\int_{\mathbb{R}} p^\prime(\bar{v})V_{xtt}^{2}{\rm d} x\\
  \leq&C(1+t)^{-1}\|V_{tt}(t)\|^{2}+C(1+t)^{-2}\|V_{xt}(t)\|_{1}^{2}+C\|V_{ttt}(t)\|^{2}+C(1+t)^{-4}\|V_{x}(t)\|_{1}^{2}\\
&+C{\rm e}^{-t}\|V_{xtt}(t)\|^{2}+C\delta(1+t)^{-\frac{11}{2}}.
  \end{split}
  \end{equation}
Next, having $\int_{\mathbb{R}}\partial_{tt}\eqref{1.23}_{1}\times V_{ttt}{\rm d}x$, we first obtain
\begin{equation}\label{3.94}
\begin{split}
  &\frac{1}{2}\frac{{\rm d}}{{\rm d}t}\int_{\mathbb{R}}(V_{ttt}^{2}-p^\prime(\bar{v})V_{xtt}^{2}{\rm d}x+\int_{\mathbb{R}}V_{ttt}^{2}{\rm d}x\\
  & = -\frac{1}{2}\int_{\mathbb{R}}p^\prime(\bar{v})_{t}V_{xtt}^{2}{\rm d}x+\int_{\mathbb{R}}\left[(F_{1tt}-\left(\left(p^\prime(\bar{v})\right)_{tt}V_{x}+\left(p^\prime(\bar{v})\right)_{t}V_{xt}\right)_{x}\right]V_{ttt}{\rm d}x\\
  &~~~~+\int_{\mathbb{R}}F_{2tt}V_{ttt}{\rm d}x .
  \end{split}
  \end{equation}
  From Lemma \ref{L2.1}, we have
  \begin{equation}\label{3.95}
  	-\frac{1}{2}\int_{\mathbb{R}}p^\prime(\bar{v})_{t}V_{xtt}^{2}{\rm d}x \leq C\delta (1+t)^{-1}\|V_{xtt}(t)\|^{2}.
  \end{equation}
  A directly calculation shows that
  \begin{equation}\label{3.96}
  \begin{split}
   &\int_{\mathbb{R}}\left[(F_{1tt}-\left(\left(p^\prime(\bar{v})\right)_{tt}V_{x}+\left(p^\prime(\bar{v})\right)_{t}V_{xt}\right)_{x}\right]V_{ttt}{\rm d}x \\&
   \leq \frac{1}{16}\|V_{ttt}(t)\|^{2}+\frac{1}{2}\frac{{\rm d}}{{\rm d}t}\int_{\mathbb{R}}[p^\prime(V_{x}+\bar{v}+\hat{v})-p^\prime(\bar{v})]V_{xtt}^{2}{\rm d}x+C\delta (1+t)^{-\frac{13}{2}}+C(1+t)^{-3}\|V_{xt}(t)\|^{2}\\&
   ~~~~+C(1+t)^{-1}\|V_{xtt}(t)\|^{2}+C(1+t)^{-2}\|V_{xxt}(t)\|^{2}+C(1+t)^{-5}\|V_{x}(t)\|^{2}\\&
   ~~~~+C(1+t)^{-4}\|V_{xx}(t)\|^{2}.
   \end{split}
  \end{equation}
  Similar calculations to \eqref{3.52}-\eqref{3.54} yields
  \begin{equation}\label{3.97}
 \begin{split}
    &\int_{\mathbb{R}}F_{2tt}V_{ttt}{\rm d}x \\&
     \leq \frac{1}{16} \|V_{ttt}(t)\|^{2}-\frac{1}{2}\frac{{\rm d}}{{\rm d}t}\int_{\mathbb{R}}gf^\prime V_{xtt}^{2}{\rm d}x+C(1+t)^{-6}\|V_{xx}(t)\|^{2}+C\delta (1+t)^{-\frac{13}{2}}\\&
    ~~~~+C(1+t)^{-4}\|V_{xxt}(t)\|^{2}+C(1+t)^{-2}\|V_{xtt}(t)\|^{2}+C(1+t)^{-5}\|V_{xt}(t)\|^{2}.
    \end{split}
\end{equation}
Substituting \eqref{3.95}-\eqref{3.97} into \eqref{3.94}, we have
\begin{equation}\label{3.98}
\begin{split}
   &\frac{1}{2}\frac{{\rm d}}{{\rm d}t}\int_{\mathbb{R}}\left[V_{ttt}^{2}+\left(gf^\prime-p^\prime(\bar{v})\right)V_{xtt}^{2}\right]{\rm d}x+\frac{3}{4}\int_{\mathbb{R}}V_{ttt}^{2}{\rm d}x \\
    \leq& \frac{1}{2}\frac{{\rm d}}{{\rm d}t}\int_{\mathbb{R}}\left[p^\prime(V_{x}+\bar{v}+\hat{v})-p^\prime(\bar{v})\right]V_{xtt}^{2}{\rm d}x +C(1+t)^{-1}\|V_{xtt}(t)\|^{2}+C\delta (1+t)^{-\frac{13}{2}}\\&
    ~~~~+C(1+t)^{-4}\|V_{xx}(t)\|^{2}+C(1+t)^{-5}\|V_{x}(t)\|^{2}+C(1+t)^{-3}\|V_{xt}(t)\|^{2}\\&
    ~~~~+C(1+t)^{-2}\|V_{xxt}(t)\|^{2}.
    \end{split}
\end{equation}
 Addition of $\lambda \cdot \eqref{3.93}$ to \eqref{3.98} $(0<\lambda\ll 1)$, one has that
  \begin{equation}\label{3.99}
   \begin{split}
   &\frac{1}{2}\frac{{\rm d}}{{\rm d}t}\int_{\mathbb{R}} \left[V_{ttt}^{2}+\lambda V_{tt}^{2}+2\lambda V_{tt}V_{ttt}+\left(gf^\prime-p^\prime(\bar{v})\right)V_{xtt}^{2}\right]{\rm d}x+\frac{1}{2}\int_{\mathbb{R}} \left(V_{ttt}^{2}-\lambda p^\prime(\bar{v})V_{xtt}^{2}\right){\rm d} x \\
   \leq& \frac{1}{2}\frac{{\rm d}}{{\rm d}t}\int_{\mathbb{R}}\left[p^\prime(V_{x}+\bar{v}+\hat{v})-p^\prime(\bar{v})\right]V_{xtt}^{2}{\rm d}x+ C\delta (1+t)^{-\frac{11}{2}}+C(1+t)^{-1}\|V_{tt}(t)\|_{1}^{2}\\&
   ~~~~+C(1+t)^{-2}\|V_{xt}(t)\|_{1}^{2}+C(1+t)^{-4}\|V_{x}(t)\|_{1}^{2}.
    \end{split}
    \end{equation}
Multiplying \eqref{3.99} by $(1+t)^{k}, ~ k=0,1,2,3,4,$ integrating over $(0,t)$, and using Lemma \ref{L3.1}-Lemma \ref{L3.6}, we can immediately obtain
\begin{equation}\label{3.100}
 \begin{split}
   (1+t)^{4}\|V_{tt}(t)\|_{1}^{2}&+(1+t)^{4}\|V_{ttt}(t)\|^{2}+\int_{0}^{t}(1+\tau)^{4}(\|V_{xtt}(\tau)\|^{2}+\|V_{ttt}(\tau)\|^{2}){\rm d}\tau\\
   \leq&  C\left(\|V_{0}\|^{2}_{3}+\|z_{0}\|_{2}^{2}+ \delta \right) .
   \end{split}
  \end{equation}
  Integrating $(1+t)^{5} \cdot \eqref{3.98}$ over $[0,t]$ yields
  \begin{equation}\label{3.101}
   (1+t)^{5}(\|V_{xtt}(t)\|^{2}+\|V_{ttt}(t)\|^{2})+\int_{0}^{t}(1+\tau)^{5}\|V_{ttt}(\tau)\|^{2}{\rm d}\tau\leq  C\left(\|V_{0}\|^{2}_{3}+\|z_{0}\|_{2}^{2}+ \delta \right).
  \end{equation}
  \end{proof}
  Combining the above two equations, one can immediately obtain \eqref{3.89}. The proof of Lemma \ref{L3.7} is completed. 
  
  From Lemmas \ref{L3.2}-Lemma \ref{L3.7}, one can easily verify that {\it a priori }assumption \eqref{3.1} is closed. Thus we have completed the proof of Proposition \ref{p1}, and obtain \eqref{1.26}-\eqref{1.27a}.

\subsection{Proof of \eqref{1.28}-\eqref{1.29a}}\label{S3.2}

Once we have obtained \eqref{1.26}-\eqref{1.27a}, we now want to give the improved decay estimates \eqref{1.28}-\eqref{1.29a}.  As pointed out in the intrudction, our analyses are quite different from \cite{Nishihara-Wang-Yang2000}.

Firstly, one can rewrite \eqref{1.23} as
\begin{equation}\label{3.102}
 \left\{\begin{array}{l}
V_{t}+p^\prime(v_{+})V_{xx}=-V_{tt}+F_{1}+F_{2}+\left[(p^\prime(v_{+})-p^\prime(\bar{v}))V_{x}\right]_{x},\\[2mm]
(V,V_{t})|_{t=0}=(V_0,z_0)(x).
 \end{array}
        \right.
\end{equation}
It is easy to see that $V(x,t)$ has the following integral representation
\begin{align}\label{3.103}
V(x,t)=&\int_{\mathbb{R}}G\left(x-y, t\right) V_{0}\left(y\right) {\rm d}y-\int_{0}^{t} \int_{\mathbb{R}}G(x-y,t-\tau) V_{\tau \tau}(y,\tau) {\rm d}y {\rm d}\tau \nonumber\\
&+\int_{0}^{t} \int_{\mathbb{R}}G(x-y,t-\tau)(F_{1}+F_{2})(y,\tau) {\rm d}y {\rm d}\tau	\nonumber\\
&+\int_{0}^{t} \int_{\mathbb{R}}G(x-y,t-\tau)\left[(p^\prime(v_{+})-p^\prime(\bar{v}))V_{y}\right]_{y}(y,\tau) {\rm d}y {\rm d}\tau,
\end{align}
where
\begin{equation}\notag
G(x,t)=\frac{1}{\sqrt{-4\pi p^\prime(v_{+})t}}\exp\left\{\frac{x^{2}}{4p^\prime(v_{+})t}\right\}.	
\end{equation}
By integration by parts with respect to $\tau$ as in \cite{Geng-Wang2016, Nishihara1997, Zhao2000},
\begin{align}\label{3.104}
&-\int_{0}^{\frac{t}{2}} \int_{\mathbb{R}}G(x-y,t-\tau) V_{\tau \tau}(y,\tau) {\rm d}y {\rm d}\tau \nonumber\\
=&-\int_{\mathbb{R}}G(x-y,t-\tau) V_{\tau}(y,\tau) {\rm d}y\big|_{\tau=0}^{\tau=\frac{t}{2}}-\int_{0}^{\frac{t}{2}} \int_{\mathbb{R}}G_{t}(x-y,t-\tau) V_{\tau}(y,\tau) {\rm d}y {\rm d}\tau \nonumber\\
=&\int_{\mathbb{R}}G\left(x-y,t\right) z_{0}(y){\rm d}y-\int_{\mathbb{R}}G\left(x-y,\frac{t}{2}\right)V_{t}\left(y,\frac{t}{2}\right){\rm d}y\nonumber\\
&-\int_{0}^{\frac{t}{2}} \int_{\mathbb{R}}G_{t}(x-y,t-\tau) V_{\tau}(y,\tau) {\rm d}y {\rm d}\tau.
\end{align}
Hence, \eqref{3.103} can be rewritten as
\begin{align}\label{3.105}
V(x,t)=&\int_{\mathbb{R}}G\left(x-y, t\right) (V_{0}+z_{0})\left(y\right) {\rm d}y-\int_{\mathbb{R}}G\left(x-y, \frac{t}{2}\right) V_{t}\left(y, \frac{t}{2}\right) {\rm d}y \nonumber\\
 &-\int_{0}^{\frac{t}{2}} \int_{\mathbb{R}}G_{t}(x-y,t-\tau) V_{\tau}(y,\tau) {\rm d}y {\rm d}\tau-\int_{\frac{t}{2}}^{t} \int_{\mathbb{R}}G(x-y,t-\tau) V_{\tau \tau}(y,\tau) {\rm d}y {\rm d}\tau \nonumber\\
 &+\int_{0}^{t} \int_{\mathbb{R}}G(x-y,t-\tau)F_{1}(y,\tau) {\rm d}y {\rm d}\tau\nonumber\\
 &+\int_{0}^{t} \int_{\mathbb{R}}G(x-y,t-\tau)\left[(p^\prime(v_{+})-p^\prime(\bar{v}))V_{y}\right]_{y}(y,\tau) {\rm d}y {\rm d}\tau\nonumber\\
 &+\int_{\frac{t}{2}}^{t} \int_{\mathbb{R}}G(x-y,t-\tau)F_{2}(y,\tau) {\rm d}y {\rm d}\tau+\int_{0}^{\frac{t}{2}} \int_{\mathbb{R}}G(x-y,t-\tau)F_{2}(y,\tau) {\rm d}y {\rm d}\tau\nonumber\\
 :=&\sum_{i=1}^{8} J_{i}(x,t).		
\end{align}
Having obtained Proposition \ref{p1} and \eqref{3.105}, now we can first deduce the following lemma.
\begin{lemma}\label{L3.8}
Under the assumptions of Theorem \ref{Thm1}, it holds that
\begin{equation}\label{3.106}
\|\partial_{x}^{k}V(t)\| \leq C(1+t)^{-\frac{1}{4}-\frac{k}{2}},\quad 0\leq k\leq 1.		
	\end{equation}	
\end{lemma}
\begin{proof}
Let's define
\begin{equation}\label{3.107}
M(t):= \sup \limits_{0 \leq s \leq t,~0 \leq k \leq 1}(1+s)^{\frac{1}{4}+\frac{k}{2}}\|\partial_{x}^{k}V(s)\|.		
	\end{equation}
Now we only need to show $M(t)$ is bounded.	Notice that
\begin{equation}\label{3.108}
\|\partial_{x}^{k}\partial_{t}^{l}G(t)\|_{L^{p}} \leq Ct^{-\frac{1}{2}(1-\frac{1}{p})-\frac{k}{2}-l},\qquad 1\leq p\leq \infty,~k,l \geq 0,	
\end{equation}
then by employing \eqref{2.6}-\eqref{2.8}, \eqref{3.1a}-\eqref{3.1b}, \eqref{3.105}, \eqref{3.108} and Hausdorff-Young's inequality, we have
\begin{equation}\label{3.109}
\left\|J_{1}(t)\right\|\leq \left\|G\left(t\right)\right\|\left\|(V_{0}+z_{0})\right\|_{L^{1}}\leq Ct^{-\frac{1}{4}},	
\end{equation}
\begin{equation}\label{3.110}
\|J_{2}(t)\|\leq \left\|G\left(\frac{t}{2}\right)\right\|_{L^{1}} \left\|V_{t}\left(\frac{t}{2}\right)\right\|\leq Ct^{-1},	
\end{equation}
\begin{align}\label{3.111}
\|J_{3}(t)\|& \leq \int_{0}^{\frac{t}{2}} \left\|G_{t}(t-\tau)\right\|_{L^{1}}\left\| V_{\tau}(\tau)\right\| {\rm d}\tau \nonumber\\
&\leq C\int_{0}^{\frac{t}{2}}(t-\tau)^{-1}(1+\tau)^{-1}{\rm d}\tau \leq Ct^{-1}\ln(1+t),	
\end{align}
\begin{equation}\label{3.112}
\|J_{4}(t)\|\leq C\int_{\frac{t}{2}}^{t} \left\|G(t-\tau)\right\|_{L^{1}}\left\|V_{\tau \tau}(\tau)\right\| {\rm d}\tau \leq C\int_{\frac{t}{2}}^{t}(1+\tau)^{-2}{\rm d}\tau \leq Ct^{-1},		
\end{equation}
\begin{equation}\label{3.113}
\begin{split}
\|J_{5}(t)\| &\leq C\int_{0}^{t}\left\|G_{x}(t-\tau)\right\|(\left\|V_{x}(\tau)\right\|^{2}+\|\bar{v}_{t}(\tau)\|_{L^{1}}+\|\hat{v}(\tau)\|_{L^{1}}){\rm d}\tau \\
&\leq C\int_{0}^{t}(t-\tau)^{-\frac{3}{4}}(1+\tau)^{-\frac{1}{2}}{\rm d}\tau\leq C\left(\int_{0}^{\frac{t}{2}}+\int_{\frac{t}{2}}^{t}\right)(t-\tau)^{-\frac{3}{4}}(1+\tau)^{-\frac{1}{2}}{\rm d}\tau\\
&\leq Ct^{-\frac{1}{4}},
\end{split}
\end{equation}
\begin{equation}\label{3.114}
\begin{split}
\|J_{6}(t)\| &\leq C\int_{0}^{t}\left\|G_{x}(t-\tau)\right\|_{L^{1}}\left\|[(p^\prime(v_{+})-p^\prime(\bar{v}))V_{y}](\tau)\right\|{\rm d}\tau \\
&\leq C\delta \int_{0}^{t}\left\|G_{x}(t-\tau)\right\|_{L^{1}}\left\|V_{y}(\tau)\right\|{\rm d}\tau\\
&\leq C\delta M(t)\int_{0}^{t}\left\|G_{x}(t-\tau)\right\|_{L^{1}}(1+\tau)^{-\frac{3}{4}}{\rm d}\tau  \\
&\leq C\delta M(t)\int_{0}^{t}(t-\tau)^{-\frac{1}{2}}(1+\tau)^{-\frac{3}{4}}{\rm d}\tau\\
&\leq C\delta M(t)\left(\int_{0}^{\frac{t}{2}}+\int_{\frac{t}{2}}^{t}\right)(t-\tau)^{-\frac{1}{2}}(1+\tau)^{-\frac{3}{4}}{\rm d}\tau\leq C\delta M(t)t^{-\frac{1}{4}},
\end{split}
\end{equation}
and
\begin{equation}\label{3.115}
\|J_{7}(t)\|\leq \int_{\frac{t}{2}}^{t}\left\|G(t-\tau)\right\|_{L^{1}}\left\|F_{2}(\tau)\right\|{\rm d}\tau \leq C\int_{\frac{t}{2}}^{t}(1+\tau)^{-\frac{5}{4}}{\rm d}\tau \leq Ct^{-\frac{1}{4}}.
\end{equation}
Now we turn to estimate $J_{8}$, notice that
\begin{align}\label{3.116}
J_{8}&=\int_{0}^{\frac{t}{2}}\int_{\mathbb{R}}G(x-y,t-\tau)F_{2}(y,\tau) {\rm d}y {\rm d}\tau	\nonumber\\
&=\int_{0}^{\frac{t}{2}}\int_{\mathbb{R}} G(x-y,t-\tau)[g^\prime f(V_{y\tau}+\bar{u}_{y}+\hat{v}_{\tau})+gf^\prime(V_{yy}+\bar{v}_{y}+\hat{v}_{y})](y,\tau) {\rm d}y {\rm d}\tau.	
\end{align}
Since $\hat{u}(x,t)$ doesn't belong to any $L^{p}$ space for $1 \leq p < \infty$, it means that $J_{8}$ is estimated quite differently from $J_{5}$. By employing \eqref{1.25b}-\eqref{1.25c}, \eqref{2.6}-\eqref{2.8}, \eqref{3.1a}-\eqref{3.1b}, \eqref{3.108} and Hausdorff-Young's inequality, we can proof
\begin{align}\label{3.117}
&\left\|\int_{0}^{\frac{t}{2}}\int_{\mathbb{R}}G(x-y,t-\tau)(g^\prime fV_{y\tau})(y,\tau) {\rm d}y {\rm d}\tau\right\| \nonumber\\
\leq& \left\|\int_{0}^{\frac{t}{2}}\int_{\mathbb{R}}G_{x}(x-y,t-\tau)(g^\prime fV_{\tau})(y,\tau) {\rm d}y {\rm d}\tau \right\|+\left\|\int_{0}^{\frac{t}{2}}\int_{\mathbb{R}}G(x-y,t-\tau)[V_{\tau}(g^{\prime } f)_{y}](y,\tau) {\rm d}y {\rm d}\tau\right\| \nonumber\\
\leq&C\int_{0}^{\frac{t}{2}}\left(\left\|G_{x}(t-\tau)\right\|_{L^{1}}\left\|(g^\prime fV_{\tau})(\tau)\right\|+\left\|G(t-\tau)\right\|\|V_{\tau}(\tau)\|\|(g^{\prime } f)_{y}(\tau)\|\right){\rm d}\tau \nonumber\\
\leq& Ct^{-\frac{1}{2}}\int_{0}^{\frac{t}{2}}(1+\tau)^{-\frac{3}{2}}{\rm d}\tau +Ct^{-\frac{1}{4}}\int_{0}^{\frac{t}{2}}(1+\tau)^{-\frac{7}{4}}{\rm d}\tau\leq Ct^{-\frac{1}{4}},
\end{align}	
and
\begin{equation}\label{3.118}
\begin{split}
&\left\|\int_{0}^{\frac{t}{2}}\int_{\mathbb{R}}G(x-y,t-\tau)(g^\prime f\hat{v}_{\tau}+gf^\prime \hat{v}_{y})(y,\tau) {\rm d}y {\rm d}\tau \right\|\\
&~~~~~\leq \int_{0}^{\frac{t}{2}}\left\|G(t-\tau)\right\|\left\|(g^\prime f\hat{v}_{\tau}+gf^\prime \hat{v}_{y})(\tau)\right\|_{L^{1}} {\rm d}y {\rm d}\tau\\
&~~~~~\leq C\int_{0}^{\frac{t}{2}}(t-\tau)^{-\frac{1}{4}}{\rm e}^{-\tau}{\rm d}\tau \leq Ct^{-\frac{1}{4}}.
\end{split} 		
\end{equation}
In a fashion similar to \eqref{3.117}, without any difficulty, we can proof
\begin{equation}\label{3.119}
\left\|\int_{0}^{\frac{t}{2}}\int_{\mathbb{R}}G(x-y,t-\tau)(gf^\prime V_{yy})(y,\tau) {\rm d}y {\rm d}\tau \right\| \leq Ct^{-\frac{1}{4}}.
\end{equation}
Finally, notice that
\begin{equation}\label{3.119a}
 \begin{split}
 &\int_{0}^{\frac{t}{2}}\int_{\mathbb{R}}G(x-y,t-\tau)(g^\prime f\bar{u}_{y})(y,\tau) {\rm d}y {\rm d}\tau \\
 =&\int_{0}^{\frac{t}{2}}\int_{\mathbb{R}}G_{x}(x-y,t-\tau)(g^\prime f\bar{u})(y,\tau) {\rm d}y {\rm d}\tau-\int_{0}^{\frac{t}{2}}\int_{\mathbb{R}}G(x-y,t-\tau)(g^{\prime\prime} fu_{y}\bar{u})(y,\tau) {\rm d}y {\rm d}\tau \\
 &-\int_{0}^{\frac{t}{2}}\int_{\mathbb{R}}G(x-y,t-\tau)(g^{\prime}f^\prime v_{y}\bar{u})(y,\tau) {\rm d}y {\rm d}\tau\\
 =&\int_{0}^{\frac{t}{2}}\int_{\mathbb{R}}G_{x}(x-y,t-\tau)(g^\prime f\bar{u})(y,\tau){\rm d}y {\rm d}\tau\\
 &-\int_{0}^{\frac{t}{2}}\int_{\mathbb{R}}G(x-y,t-\tau)\left\{\bar{u}[(g^{\prime\prime} f(V_{y\tau}+\hat{v}_{\tau})+g^{\prime}f^\prime (V_{yy}+\hat{v}_{y})]\right\}(y,\tau) {\rm d}y {\rm d}\tau \\
 &-\int_{0}^{\frac{t}{2}}\int_{\mathbb{R}}G(x-y,t-\tau)(g^{\prime\prime} f\bar{u}_{y}\bar{u})(y,\tau) {\rm d}y {\rm d}\tau-\int_{0}^{\frac{t}{2}}\int_{\mathbb{R}}G(x-y,t-\tau)(g^{\prime}f^\prime \bar{v}_{y}\bar{u})(y,\tau) {\rm d}y {\rm d}\tau,
  \end{split}
\end{equation}
and
\begin{equation}\label{3.119b}
  \begin{split}
 &-\int_{0}^{\frac{t}{2}}\int_{\mathbb{R}}G(x-y,t-\tau)(g^{\prime\prime} f\bar{u}_{y}\bar{u})(y,\tau) {\rm d}y {\rm d}\tau\\
 =&-\frac{1}{2}\int_{0}^{\frac{t}{2}}\int_{\mathbb{R}}G(x-y,t-\tau)[g^{\prime\prime} f(\bar{u}^{2})_{y}](y,\tau) {\rm d}y {\rm d}\tau\\
 =&-\frac{1}{2}\int_{0}^{\frac{t}{2}}\int_{\mathbb{R}}G_{x}(x-y,t-\tau)(g^{\prime\prime} f\bar{u}^{2})(y,\tau) {\rm d}y {\rm d}\tau+\frac{1}{2}\int_{0}^{\frac{t}{2}}\int_{\mathbb{R}}G(x-y,t-\tau)(g^{\prime\prime\prime} f\bar{u}^{2}u_{y})(y,\tau) {\rm d}y {\rm d}\tau\\
 &+\frac{1}{2}\int_{0}^{\frac{t}{2}}\int_{\mathbb{R}}G(x-y,t-\tau)[g^{\prime\prime}f^\prime\bar{u}^{2}(V_{yy}+\hat{v}_{y})](y,\tau) {\rm d}y {\rm d}\tau\\
 &+\frac{1}{2}\int_{0}^{\frac{t}{2}}\int_{\mathbb{R}}G(x-y,t-\tau)(g^{\prime\prime}f^\prime\bar{u}^{2}\bar{v}_{y})(y,\tau) {\rm d}y {\rm d}\tau.
  \end{split}
\end{equation}
Then it follows that
\begin{equation}\label{3.119c}
 \begin{split}
 &\int_{0}^{\frac{t}{2}}\int_{\mathbb{R}}G(x-y,t-\tau)(g^\prime f\bar{u}_{y}+gf^\prime \bar{v}_{y})(y,\tau) {\rm d}y {\rm d}\tau \\
 =&\int_{0}^{\frac{t}{2}}\int_{\mathbb{R}}G_{x}(x-y,t-\tau)\left(g^\prime f\bar{u}-\frac{1}{2}g^{\prime\prime} f\bar{u}^{2}\right)(y,\tau) {\rm d}y {\rm d}\tau \\
 &-\int_{0}^{\frac{t}{2}}\int_{\mathbb{R}}G(x-y,t-\tau)\left\{\bar{u}[(g^{\prime\prime} f(V_{y\tau}+\hat{v}_{\tau})+g^{\prime}f^\prime (V_{yy}+\hat{v}_{y})]\right\}(y,\tau) {\rm d}y {\rm d}\tau \\
 &+\frac{1}{2}\int_{0}^{\frac{t}{2}}\int_{\mathbb{R}}G(x-y,t-\tau)\{\bar{u}^{2}[g^{\prime\prime\prime} fu_{y}+g^{\prime\prime}f^\prime(V_{yy}+\hat{v}_{y})]\}(y,\tau) {\rm d}y {\rm d}\tau \\
 &+\int_{0}^{\frac{t}{2}}\int_{\mathbb{R}}G(x-y,t-\tau)\left[\left(\frac{1}{2}g^{\prime\prime}\bar{u}^{2}+g-g^{\prime}\bar{u}\right)f^\prime \bar{v}_{y}\right](y,\tau) {\rm d}y {\rm d}\tau\\
 :=&K_{1}(x,t)+K_{2}(x,t)+K_{3}(x,t)+K_{4}(x,t).
  \end{split}
\end{equation}
By employing $\eqref{1.6}_{2}$, \eqref{1.25c}, \eqref{2.6}-\eqref{2.7}, \eqref{3.1a}-\eqref{3.1b} and \eqref{3.108}, we can obtain
\begin{equation}\label{3.119d}
  \left\|K_{1}(t)\right\|\leq C\int_{0}^{\frac{t}{2}}\left\|G_{x}(t-\tau)\right\|_{L^{1}}(\left\|(u\bar{u})(\tau)\right\|+\left\|\bar{u}^{2}(\tau)\right\|){\rm d}\tau\\
  \leq C\int_{0}^{\frac{t}{2}}(t-\tau)^{-\frac{1}{2}}(1+\tau)^{-\frac{3}{4}}{\rm d}\tau\leq Ct^{-\frac{1}{4}},
\end{equation}
and
\begin{equation}\label{3.119e}
 \begin{split}
 &\left\|K_{2}(t)\right\|+\left\|K_{3}(t)\right\| \\
 \leq&C\int_{0}^{\frac{t}{2}}\left\|G(t-\tau)\right\|(1+\tau)^{-\frac{3}{2}}\left\|\bar{u}(\tau)\right\|_{L^{1}}{\rm d}\tau\leq C\int_{0}^{\frac{t}{2}}(t-\tau)^{-\frac{1}{4}}(1+\tau)^{-\frac{3}{2}}{\rm d}\tau\leq Ct^{-\frac{1}{4}}.
  \end{split}
\end{equation}
As for $K_{4}(x,t)$, firstly, it follows from Taylor's expansion that
\begin{equation}\notag
 g(u)=\frac{1}{2}g^{\prime\prime}(\theta_{1}u)u^{2},~~~ g^\prime(u)=g^{\prime\prime}(\theta_{2}u)u,
\end{equation}
for $0<\theta_{1}, \theta_{2}<1$. Noticing that $u=V_{t}+\bar{u}+\hat{u}$, it is easy to obtain that
\begin{equation}\label{3.119g}
\begin{split}
 &\left\|\left(\frac{1}{2}g^{\prime\prime}(u)\bar{u}^{2}+g(u)-g^{\prime}(u)\bar{u}\right)(t)\right\|_{L^{\infty}}\\
 =&\left\|\left(\frac{1}{2}g^{\prime\prime}(u)\bar{u}^{2}+\frac{1}{2}g^{\prime\prime}(\theta_{1}u)u^{2}-g^{\prime\prime}(\theta_{2}u)u\bar{u}\right)(t)\right\|_{L^{\infty}}\leq C(1+t)^{-\frac{3}{2}},
 \end{split}
\end{equation}
thus, one can immediately obtain
\begin{equation}\label{3.120}
 \left\|K_{4}(t)\right\|\leq C\int_{0}^{\frac{t}{2}}\left\|G(t-\tau)\right\|\left\|\bar{v}_{y}(\tau)\right\|_{L^{1}}(1+\tau)^{-\frac{3}{2}}{\rm d}\tau \leq C\int_{0}^{\frac{t}{2}}(t-\tau)^{-\frac{1}{4}}(1+\tau)^{-\frac{3}{2}}{\rm d}\tau\leq Ct^{-\frac{1}{4}}.
\end{equation}
Then, from \eqref{3.117}-\eqref{3.119}, \eqref{3.119c}-\eqref{3.119e} and \eqref{3.120}, we get
\begin{equation}\label{3.121}
\|J_{8}(t)\| \leq Ct^{-\frac{1}{4}}.	
\end{equation}
Consequently,
\begin{equation}\label{3.122}
\|V(t)\| \leq \sum_{i=1}^{8}\|J_{i}(t)\|\leq C(1+\delta M(t))(1+t)^{-\frac{1}{4}}.	
\end{equation}

With the above preparations in hand, we now turn to prove \eqref{3.106}.  Firstly, by combintion of \eqref{3.8}-\eqref{3.11} and \eqref{3.122}, it is easy to check that
\begin{equation}\label{3.123}
\int_{\mathbb{R}}F_{2}V{\rm d} x\leq C\|V_{t}(t)\|^{2}+C{\rm e}^{-t}\|V(t)\|_{2}^{2}+ C\delta (1+t)^{-\frac{3}{2}}+C\delta^{2}M^{2}(t)(1+t)^{-\frac{3}{2}}.	
\end{equation}
Combining \eqref{3.5}-\eqref{3.6} with \eqref{3.123}, we have
\begin{equation}\label{3.124}
\begin{split}
 &\frac{{\rm d}}{{\rm d}t}\int_{\mathbb{R}} \left(\frac{V^{2}}{2}+VV_{t}\right){\rm d} x-\frac{3}{4}\int_{\mathbb{R}} p^\prime(\bar{v})V_{x}^{2}{\rm d}x\\
 & \leq C\|V_{t}(t)\|^{2}+C{\rm e}^{-t}\|V(t)\|_{2}^{2}+C\delta (1+t)^{-\frac{3}{2}}+C\delta^{2}M^{2}(t)(1+t)^{-\frac{3}{2}}.
 \end{split}
\end{equation}
Addition of $\lambda \cdot \eqref{3.124}$, $0<\lambda \ll 1$ to \eqref{3.19} yields
  \begin{align}\label{3.125}
    &\frac{1}{2}\frac{{\rm d}}{{\rm d}t}\int_{\mathbb{R}} \left(V_{t}^{2}+\lambda V^{2}+2\lambda VV_{t}-p^\prime(\bar{v})V_{x}^{2}\right){\rm d} x+\frac{1}{2}\int_{\mathbb{R}}(V_{t}^{2}-\lambda p^\prime(\bar{v})V_{x}^{2}){\rm d}x\nonumber\\
     \leq& C{\rm e}^{-t}\|V(t)\|_{2}^{2}+\frac{{\rm d}}{{\rm d}t}\int_{\mathbb{R}}\left[\int_{\bar{v}}^{V_{x}+\bar{v}+\hat{v}}p(s){\rm d}s-p(\bar{v})V_{x}-\frac{p^\prime(\bar{v})}{2}V_{x}^{2}\right]{\rm d}x\nonumber\\
     &+C\delta (1+t)^{-\frac{3}{2}}+C\delta^{2}M^{2}(t)(1+t)^{-\frac{3}{2}}.
  \end{align}
Next, integrating $(1+t)^{\epsilon_{0}+\frac{1}{2}}\times \eqref{3.125}$ over $(0, t)$ for any fixed $0<\epsilon_{0}<\frac{1}{2}$, we obtain
\begin{align}\notag
  &\frac{1}{2}(1+t)^{\epsilon_{0}+\frac{1}{2}}\int_{\mathbb{R}} \left(V_{t}^{2}+\lambda V^{2}+2\lambda VV_{t}-p^\prime(\bar{v})V_{x}^{2}\right){\rm d} x\nonumber\\
  &+\frac{1}{2}\int_{0}^{t}\int_{\mathbb{R}} (1+\tau)^{\epsilon_{0}+\frac{1}{2}}\left(V_{t}^{2}-\lambda p^\prime(\bar{v}) V_{x}^{2} \right){\rm d} x{\rm d}\tau \nonumber\\
  \leq&C\int_{0}^{t}(1+\tau)^{\epsilon_{0}-\frac{1}{2}}(\|V(\tau)\|_{1}^{2}+\|V_{t}(\tau)\|^{2}){\rm d}\tau+C(\varepsilon +\delta)(1+t)^{\epsilon_{0}+\frac{1}{2}}\|V_{x}(t)\|^{2}\nonumber\\
  &+C\int_{0}^{t}(1+\tau)^{\epsilon_{0}+\frac{1}{2}}{\rm e}^{-\tau}\|V(\tau)\|_{2}^{2}{\rm d}\tau +C(1+\delta^{2} M^{2}(t))(1+t)^{\epsilon_{0}}.	
  \end{align}
  By using \eqref{3.1a}-\eqref{3.1b} and \eqref{3.122}, it is easy to obtain that
  \begin{align}\notag
  &\int_{0}^{t}(1+\tau)^{\epsilon_{0}-\frac{1}{2}}(\|V(\tau)\|_{1}^{2}+\|V_{t}(\tau)\|^{2}){\rm d}\tau \nonumber\\
  &\leq C(1+\delta^{2}M^{2}(t))\int_{0}^{t}(1+\tau)^{\epsilon_{0}-1}{\rm d}\tau+C\int_{0}^{t}(1+\tau)^{\epsilon_{0}-\frac{1}{2}}[(1+\tau)^{-1}+(1+\tau)^{-2}]{\rm d}\tau\nonumber\\
  &\leq C(1+\delta^{2}M^{2}(t))(1+t)^{\epsilon_{0}}+C\leq C(1+\delta^{2}M^{2}(t))(1+t)^{\epsilon_{0}},\nonumber
  \end{align}
  and
  \begin{equation}\notag
  \int_{0}^{t}(1+\tau)^{\epsilon_{0}+\frac{1}{2}}{\rm e}^{-\tau}\|V(\tau)\|_{2}^{2}{\rm d}\tau \leq \sup \limits_{0 \leq \tau \leq t}\|V(\tau)\|_{2}^{2}\int_{0}^{t}(1+\tau)^{\epsilon_{0}+\frac{1}{2}}{\rm e}^{-\tau}{\rm d}\tau\leq C.	
  \end{equation}
  Notice that $\varepsilon$ and $\delta$ are small enough, then it follows that
  \begin{align}\label{3.126}
(1+t)^{\epsilon_{0}+\frac{1}{2}}&(\|V(t)\|_{1}^{2}+\|V_{t}(t)\|^{2})+\int_{0}^{t}(1+\tau)^{\epsilon_{0}+\frac{1}{2}}\left(\|V_{x}(\tau)\|^{2}+\|V_{t}(\tau)\|^{2}\right){\rm d}\tau\nonumber\\
& \leq C(1+\delta^{2}M^{2}(t))(1+t)^{\epsilon_{0}}.	
  \end{align}
 Then, the integration of $(1+t)^{\epsilon_{0}+\frac{3}{2}}\times \eqref{3.19}$ over $(0, t)$, we obtain
  \begin{align}\notag
  &\frac{1}{2}(1+t)^{\epsilon_{0}+\frac{3}{2}}\int_{\mathbb{R}}\left(V_{t}^{2}-p^\prime(\bar{v})V_{x}^{2}\right){\rm d} x+\frac{3}{4}\int_{0}^{t}\int_{\mathbb{R}}(1+\tau)^{\epsilon_{0}+\frac{3}{2}} V_{t}^{2}{\rm d} x{\rm d}\tau\nonumber\\
  \leq &C\int_{0}^{t}(1+\tau)^{\epsilon_{0}+\frac{1}{2}}(\|V_{x}(\tau)\|^{2}+\|V_{t}(\tau)\|^{2}){\rm d}\tau+C(\varepsilon +\delta)(1+t)^{\epsilon_{0}+\frac{3}{2}}\|V_{x}(t)\|^{2}\nonumber\\
  &+C\int_{0}^{t}(1+\tau)^{\epsilon_{0}+\frac{3}{2}}{\rm e}^{-\tau}\|V_{xx}(\tau)\|^{2}{\rm d}\tau+C(1+t)^{\epsilon_{0}}.\nonumber
  \end{align}
  By using \eqref{3.1a}-\eqref{3.1b} and \eqref{3.126}, it is easy to obtain that
  \begin{equation}\label{3.127}
  (1+t)^{\epsilon_{0}+\frac{3}{2}}(\|V_{x}(t)\|^{2}+\|V_{t}(t)\|^{2})+\int_{0}^{t}(1+\tau)^{\epsilon_{0}+\frac{3}{2}}\|V_{t}(\tau)\|^{2}{\rm d}\tau\leq C(1+\delta^{2}M^{2}(t))(1+t)^{\epsilon_{0}}.
  \end{equation}
It follows from \eqref{3.126} and \eqref{3.127} that
\begin{equation}\notag
\sum \limits_{k=0}^{1}(1+t)^{\frac{1}{2}+k}\|\partial_{x}^{k}V(t)\|^{2} \leq C+C\delta^{2}M^{2}(t).
\end{equation}
Thus, one can immediately obtain
\begin{equation}\notag
M^{2}(t)\leq C+C\delta^{2}M^{2}(t).	
\end{equation}
Since $\delta$ is sufficiently small, we have
\begin{equation}\notag
M^{2}(t)\leq C,	
\end{equation}
which implies \eqref{3.106}. Now we have completed the proof of Lemma \ref{L3.8}.

Having obtained Lemma \ref{L3.8}, combining Proposition \ref{p1}, now we turu to proof the following lemma.
\begin{lemma}\label{L3.9}
 Under the assumptions in Theorem \ref{Thm1}, then we have $V(x, t)$ satisfies the following decay estimates:
  \begin{align}
&\label{3.128}\|\partial_{x}^{k}\partial_{t}^{l}V(t)\| \leq C(1+t)^{-\frac{1}{4}-\frac{k}{2}-l},\quad 0\leq k+l \leq 3,~0\leq l \leq 2,\\
&\label{3.129}\|\partial_{t}^{3}V(t)\|\leq C(1+t)^{-\frac{11}{4}}.	
\end{align}  	
       \end{lemma}
 \begin{proof}
  Firstly, from Lemma \ref{L3.8}, one can easily verify that
  \begin{align}\label{3.130}
 (1+t)^{\epsilon_{0}+\frac{1}{2}}(\|V(t)\|_{1}^{2}+\|V_{t}(t)\|^{2})+\int_{0}^{t}(1+\tau)^{\epsilon_{0}+\frac{1}{2}}\left(\|V_{x}(\tau)\|^{2}+\|V_{t}(\tau)\|^{2}\right){\rm d}\tau \leq C(1+t)^{\epsilon_{0}},	
  \end{align}
 and
  \begin{equation}\label{3.131}
  (1+t)^{\epsilon_{0}+\frac{3}{2}}(\|V_{x}(t)\|^{2}+\|V_{t}(t)\|^{2})+\int_{0}^{t}(1+\tau)^{\epsilon_{0}+\frac{3}{2}}\|V_{t}(\tau)\|^{2}{\rm d}\tau\leq C(1+t)^{\epsilon_{0}},
  \end{equation}
for any fixed $0<\epsilon_{0}<\frac{1}{2}$.
  Next, integrating $(1+t)^{\epsilon_{0}+\frac{3}{2}}\times \eqref{3.31}$ over $(0, t)$, we can obtain
  \begin{align}\notag
  &\frac{1}{2}(1+t)^{\epsilon_{0}+\frac{3}{2}}\int_{\mathbb{R}} \left[V_{xt}^{2}+\lambda V_{x}^{2}+2\lambda V_{xt}V_{x}+\left(-p^\prime(\bar{v})+gf^\prime\right)V_{xx}^{2}\right]{\rm d} x\nonumber\\
  &+\frac{1}{2}\int_{0}^{t}\int_{\mathbb{R}}(1+\tau)^{\epsilon_{0}+\frac{3}{2}}\left(V_{xt}^{2}-\lambda p^\prime(\bar{v})V_{xx}^{2}\right){\rm d}x{\rm d}\tau \nonumber\\
   \leq& C\int_{0}^{t}(1+\tau)^{\epsilon_{0}+\frac{1}{2}}(\|V_{x}(\tau)\|_{1}^{2}+\|V_{xt}(\tau)\|^{2}){\rm d}\tau+C(\varepsilon +\delta)(1+t)^{\epsilon_{0}+\frac{3}{2}}\|V_{xx}(t)\|^{2}\nonumber\\
  &+C\int_{0}^{t}(1+\tau)^{\epsilon_{0}+\frac{3}{2}}{\rm e}^{-\tau}\|V_{xx}(\tau)\|^{2}{\rm d}\tau+C(1+t)^{\epsilon_{0}}.\nonumber
  \end{align}
  By using \eqref{3.1a}-\eqref{3.1b} and \eqref{3.130}, we have
  \begin{equation}\notag
  \int_{0}^{t}(1+\tau)^{\epsilon_{0}+\frac{1}{2}}(\|V_{x}(\tau)\|_{1}^{2}+\|V_{xt}(\tau)\|^{2}){\rm d}\tau \leq C(1+t)^{\epsilon_{0}}.	
  \end{equation}
  Then it follows that
\begin{equation}\label{3.132}
  (1+t)^{\epsilon_{0}+\frac{3}{2}}(\|V_{x}(t)\|_{1}^{2}+\|V_{xt}(t)\|^{2})+\int_{0}^{t}(1+\tau)^{\epsilon_{0}+\frac{3}{2}}(\|V_{xx}(\tau)\|^{2}+\|V_{xt}(\tau)\|^{2}){\rm d}\tau \leq C(1+t)^{\epsilon_{0}}.	
  \end{equation}
  Integrating $(1+t)^{\epsilon_{0}+\frac{5}{2}}\times \eqref{3.30}$ over $(0, t)$, we obtain
  \begin{align}\notag
  &\frac{1}{2}(1+t)^{\epsilon_{0}+\frac{5}{2}}\int_{\mathbb{R}}\left(V_{xt}^{2}+\left(-p^\prime(\bar{v})+gf^\prime\right)V_{xx}^{2}\right){\rm d} x+\frac{3}{4}\int_{0}^{t}\int_{\mathbb{R}}(1+\tau)^{\epsilon_{0}+\frac{5}{2}} V_{xt}^{2}{\rm d} x{\rm d}\tau\nonumber\\
  \leq &C\int_{0}^{t}(1+\tau)^{\epsilon_{0}+\frac{3}{2}}(\|V_{xx}(\tau)\|^{2}+\|V_{xt}(\tau)\|^{2}){\rm d}\tau+C\int_{0}^{t}(1+\tau)^{\epsilon_{0}+\frac{1}{2}}\|V_{x}(\tau)\|^{2}{\rm d}\tau\nonumber\\
  &+C(\varepsilon +\delta)(1+t)^{\epsilon_{0}+\frac{5}{2}}\|V_{xx}(t)\|^{2}+C\int_{0}^{t}(1+\tau)^{\epsilon_{0}+\frac{5}{2}}{\rm e}^{-\tau}\|V_{xx}(\tau)\|^{2}{\rm d}\tau+C(1+t)^{\epsilon_{0}}.\nonumber	
  \end{align}
  By using \eqref{3.1a}-\eqref{3.1b}, \eqref{3.130} and \eqref{3.132}, we can immediately obtain
\begin{equation}\label{3.133}
  (1+t)^{\epsilon_{0}+\frac{5}{2}}(\|V_{xx}(t)\|^{2}+\|V_{xt}(t)\|^{2})+\int_{0}^{t}(1+\tau)^{\epsilon_{0}+\frac{5}{2}}\|V_{xt}(\tau)\|^{2}{\rm d}\tau \leq C(1+t)^{\epsilon_{0}}.	
  \end{equation}
 In a similar fashion as above, integrating $(1+t)^{\epsilon_{0}+\frac{5}{2}}\times \eqref{3.56}$ and $(1+t)^{\epsilon_{0}+\frac{7}{2}}\times \eqref{3.55}$ over $(0, t)$, we obtain
 \begin{equation}\label{3.134}
  (1+t)^{\epsilon_{0}+\frac{5}{2}}(\|V_{xx}(t)\|_{1}^{2}+\|V_{xxt}(t)\|^{2})+\int_{0}^{t}(1+\tau)^{\epsilon_{0}+\frac{5}{2}}(\|V_{xxx}(\tau)\|^{2}+\|V_{xxt}(\tau)\|^{2}){\rm d}\tau \leq C(1+t)^{\epsilon_{0}},	
  \end{equation}
\begin{equation}\label{3.135}
  (1+t)^{\epsilon_{0}+\frac{7}{2}}(\|V_{xxx}(t)\|^{2}+\|V_{xxt}(t)\|^{2})+\int_{0}^{t}(1+\tau)^{\epsilon_{0}+\frac{7}{2}}\|V_{xxt}(\tau)\|^{2}{\rm d}\tau \leq C(1+t)^{\epsilon_{0}}.	
  \end{equation}
Integrating $(1+t)^{\epsilon_{0}+\frac{5}{2}}\times \eqref{3.77}$ and $(1+t)^{\epsilon_{0}+\frac{7}{2}}\times \eqref{3.76}$ over $(0, t)$, we obtain
\begin{equation}\label{3.136}
  (1+t)^{\epsilon_{0}+\frac{5}{2}}(\|V_{t}(t)\|_{1}^{2}+\|V_{tt}(t)\|^{2})+\int_{0}^{t}(1+\tau)^{\epsilon_{0}+\frac{5}{2}}(\|V_{xt}(\tau)\|^{2}+\|V_{tt}(\tau)\|^{2}){\rm d}\tau \leq C(1+t)^{\epsilon_{0}},	
  \end{equation}
\begin{equation}\label{3.137}
  (1+t)^{\epsilon_{0}+\frac{7}{2}}(\|V_{xt}(t)\|^{2}+\|V_{tt}(t)\|^{2})+\int_{0}^{t}(1+\tau)^{\epsilon_{0}+\frac{7}{2}}\|V_{tt}(\tau)\|^{2}{\rm d}\tau \leq C(1+t)^{\epsilon_{0}}.	
  \end{equation}
Integrating $(1+t)^{\epsilon_{0}+\frac{9}{2}}\times \eqref{3.88}$ over $(0, t)$, we obtain
\begin{equation}\label{3.138}
  (1+t)^{\epsilon_{0}+\frac{9}{2}}(\|V_{xxt}(t)\|^{2}+\|V_{xtt}(t)\|^{2})+\int_{0}^{t}(1+\tau)^{\epsilon_{0}+\frac{9}{2}}\|V_{xtt}(\tau)\|^{2}{\rm d}\tau \leq C(1+t)^{\epsilon_{0}}.	
  \end{equation}
  Integrating $(1+t)^{\epsilon_{0}+\frac{9}{2}}\times \eqref{3.99}$ and $(1+t)^{\epsilon_{0}+\frac{11}{2}}\times \eqref{3.98}$ over $(0, t)$, we obtain
  \begin{equation}\label{3.139}
  (1+t)^{\epsilon_{0}+\frac{9}{2}}(\|V_{tt}(t)\|_{1}^{2}+\|V_{ttt}(t)\|^{2})+\int_{0}^{t}(1+\tau)^{\epsilon_{0}+\frac{9}{2}}(\|V_{xtt}(\tau)\|^{2}+\|V_{ttt}(\tau)\|^{2}){\rm d}\tau \leq C(1+t)^{\epsilon_{0}},	
  \end{equation}
  and
\begin{equation}\label{3.140}
  (1+t)^{\epsilon_{0}+\frac{11}{2}}(\|V_{xtt}(t)\|^{2}+\|V_{ttt}(t)\|^{2})+\int_{0}^{t}(1+\tau)^{\epsilon_{0}+\frac{11}{2}}\|V_{ttt}(\tau)\|^{2}{\rm d}\tau \leq C(1+t)^{\epsilon_{0}}.	
  \end{equation}
Hence, from \eqref{3.130}-\eqref{3.140}, we can immediately obtain the desired estimates \eqref{3.128}-\eqref{3.129}. The proof of Lemma \ref{L3.9} is completed.
\end{proof}
Combining Lemma \ref{L3.8} with Lemma \ref{L3.9}, one can immediately obtain \eqref{1.28}-\eqref{1.29a} in our main Theorem \ref{Thm1}.
\end{proof}

\vspace{6mm}

 \noindent {\bf Acknowledgements:} The research was supported by the National Natural Science Foundation of China $\#$12171160, 11771150, 11831003 and Guangdong Basic and Applied Basic Research Foundation $\#$2020B1515310015.
\bigbreak

%
%
%
{\small

\bibliographystyle{plain}

}

\end{document}